\documentclass{amsart}
\usepackage{amsmath,amssymb, amsthm, tikz}
\usepackage{pgfkeys}
\usepackage{graphics}
\usepackage{enumitem}

\newtheorem{lemm}{Lemma}[section]
\newtheorem{theorem}{Theorem}[section]
\newtheorem{proposition}{Proposition}[section]
\newtheorem{lemma}{Lemma}[section]
\newtheorem{cor}{Corollary}[section]

\newtheorem{propos}[lemm]{Proposition}

\newenvironment{defi}{\medskip\noindent{\sc
Definition}. }{\goodbreak\medskip}

\newenvironment{nota}{\medskip\noindent{\sc
Notation}.}{\goodbreak\medskip}
\newenvironment{remk}{\noindent{\sc
Remark}. }{\goodbreak\vskip10pt}
\newenvironment{remks}{\noindent{\sc
Remarks}. }{\goodbreak\vskip10pt}
\newenvironment{notas}{\medskip\noindent{\sc
Notations}. }{\goodbreak\medskip}

\newenvironment{exa}{\noindent{\sc
Example}. }{\goodbreak\vskip10pt}
\newenvironment{ques}{\noindent{\sc
Question}. }{\goodbreak\vskip10pt}

\def\cb{{\mathcal B}}
\def\cF{{\mathcal F}}
\def\cs{{\mathcal S}}
\def\ct{{\mathcal T}}

\def\cC{{\mathcal C}}

\def\ci{{\mathcal I}}

\def\cg{{\mathcal G}}
\def\ca{{\mathcal A}}
\def\ch{{\mathcal H}}
\def\ck{{\mathcal K}}

\def\cm{{\mathcal M}}
\def\cN{{\mathcal N}}

\def\cu{{\mathcal U}}

\def\cv{{\mathcal V}}

\def\R{\mathbb{R}}
\def\A{\mathbb{A}}
\def\Z{\mathbb{Z}}
\def\N{\mathbb{N}}

\def\C{\mathbb{C}}
\def\T{\mathbb{T}}
\def\Q{\mathbb{Q}}

\def\smallskip{\par\vspace{1mm}}
\def\medskip{\par\vspace{2mm}}
\def\bigskip{\par\vspace{3mm}}

\newcount\equationcount
\def\thenumber{0}
\def\eq#1{\global\advance\equationcount by 1
   \def\thenumber{\number\equationcount}
                        {$$#1\eqno(\thenumber)$$}}

\tikzset{
xmin/.store in=\xmin, xmin/.default=-1.5, xmin=-1.5,
xmax/.store in=\xmax, xmax/.default=7.5, xmax=7.55,
ymin/.store in=\ymin, ymin/.default=-0.75, ymin=-0.75,
ymax/.store in=\ymax, ymax/.default=3.25, ymax=3.25,
}

\begin{document}

\title[transversal weak K.A.M.]{On the transversal dependence of weak K.A.M. solutions for symplectic twist maps}

\author{Marie-Claude Arnaud$^{\dag,\ddag}$, Maxime Zavidovique$^{*,**}$}

\email{Marie-Claude.Arnaud@univ-avignon.fr \\ maxime.zavidovique@upmc.fr}

\date{}



\thanks{$\dag$ Avignon Universit\'e, Laboratoire de Math\'ematiques d'Avignon 
(EA 2151)\\ F-84018  Avignon, FRANCE } 
\thanks{$\ddag$ member of the {\sl Institut universitaire de France.}}
\thanks{ $*$ IMJ-PRG, UPMC
4 place Jussieu,
Case 247
75252 Paris Cedex 5}
\thanks{ $**$ financ\' e par une bourse PEPS du CNRS}

\begin{abstract}  For a symplectic twist map, we prove that there is a choice of weak K.A.M. solutions that depend in a continuous way on the cohomology class. We thus obtain a continuous function $u(\theta, c)$ in two variables: the angle $\theta$ and the cohomology class $c$.  As a result, we prove that the Aubry-Mather sets are contained in pseudographs that are vertically ordered by their rotation numbers. Then we characterize the $C^0$ integrable twist maps in terms of regularity of $u$ that allows to see $u$ as a generating function. We also obtain some results for the Lipschitz integrable twist maps. With an example, we show that our choice is not the so-called discounted one (see   \cite{DFIZ2}), that is sometimes discontinuous.  We also provide examples of `strange' continuous foliations that cannot be straightened by a symplectic homeomorphism.
\end{abstract}

\maketitle
\section{Introduction and Main Results.}\label{SecIntro}
 
Weak K.A.M. theory was developed by A.~Fathi (see \cite{Fa2} and the still unpublished book \cite{Fa3})  at the end of the 90s in the setting of autonomous Hamiltonian systems. It was quickly extended to the discrete time or periodic  time dependent setting (see for example  \cite{Be1}, \cite{Be2}, \cite{CISM} and \cite{GT1}). In the case of symplectic twist maps,  this enhances the famous Aubry-Mather theory that was created independently by S.~Aubry \& A.~Le Daeron and J.~Mather in the 80s in the case of twist maps (see \cite{ALD}, \cite{Mat1}  and also  the book \cite{Gol1}).

Weak K.A.M.  theory provides some weak solutions of variational problems (in the Lagrangian setting) or Hamilton-Jacobi equations (in the Hamiltonian setting). This also gives some special negatively or positively invariant sets called pseudographs (see \cite{Be2} or \cite{Arna2}) and some special (minimizing) invariant Borel probability measures (see \cite{Mat2}).

To be more precise, the theory gives a family of solutions i.e. functions in $C^0(M, \R)$ that are Lipschitz continuous and even semi-concave\footnote{The definition of a semi-concave function is given in subsection \ref{ssweakk}.}, at least one for each cohomology class of $H^1(M, \R)$ if we work on the manifold $M$. Then the two following (related)  questions are natural:
\begin{enumerate}
\item\label{Pt1} is there some natural choice of a weak K.A.M. solution in any cohomology class?
\item\label{Pt2}  does there exist a choice that is transversely regular ($C^0$, $C^1$, smooth...) with respect to the cohomology class (with some hypothesis  on the considered twist map)?

\end{enumerate}
Of course, point (\ref{Pt2}) is related to point (\ref{Pt1}) because if we have not a unique choice of a weak K.A.M. solution for  every cohomology class $c\in H^1(M, \R)$, we cannot speak of $C^1$ regularity with respect to $c$ for the map $c\mapsto \{u_c\}$ that sends $c$ to the whole set of weak K.A.M. solutions of cohomology class $c$. Observe nevertheless that a kind of local Lipschitz regularity was studied in \cite{LYY} (for weak K.A.M. solutions for Tonelli Hamiltonians) with no uniqueness.
\\
An answer to question (\ref{Pt1}) (called discounted solution) was recently given in \cite{DFIZ1} in the autonomous case and in  \cite{DFIZ2} and \cite{SuThi} in the discrete case. Here we will focus on question (\ref{Pt2}) in the case of symplectic twist maps.

The first statement is that we can always choose the weak K.A.M. solutions in a transversely continuous way. We begin by recalling the definition of twist map.
\begin{defi}
A {\em symplectic  twist map} or {\em symplectic twist diffeomorphism} $f:\A\rightarrow \A$ is a $C^1$ diffeomorphism such that
\begin{itemize}
\item $f$ is isotopic to identity;
\item $f$ is exact symplectic, i.e. if $f(\theta, r)=(\Theta, R)$, then the 1-form $Rd\Theta-rd\theta$ is exact;
\item $f$ has the {\em twist property} i.e. if $F=(F_1, F_2):\R^2\rightarrow \R^2$ is any lift of $F$, for any $\tilde \theta\in \R$, the map $r\in \R\mapsto F_1(\tilde \theta, r)\in \R$ is an increasing $C^1$ diffeomorphism from $\R$ onto $\R$.
\end{itemize}
\end{defi}

\begin{theorem}\label{Tgenecont}
Let $f$ be a $C^1$ symplectic twist diffeomorphism of $\T\times \R$. Then there exists a continuous map $u:\T\times \R\rightarrow \R$ such that
\begin{itemize}
\item $u(0,c)=0$;
\item each $u_c= u(\cdot ,c)$ is a weak K.A.M. solution for the cohomology class $c$, this implies that:
\begin{itemize}
\item each $u_c= u(\cdot ,c)$ is semi-concave (hence  derivable almost everywhere)\footnote{The definition of a semi-concave function is given in subsection \ref{ssweakk}.};
 \item each   partial graph of $c+\frac{\partial u_c}{\partial \theta}$ is backward invariant by $f$.
 \end{itemize}
\end{itemize}
\end{theorem}
 \begin{remks}\begin{enumerate}
 \item The choice of solution we   do is not the same than the discounted one given in \cite{DFIZ2} and \cite{SuThi}: in Appendix  \ref{sscomplint}, we will  give an example for which  the choice given in \cite{DFIZ2} and \cite{SuThi} is not continuous.
 \item In appendix A of \cite{Be2}, it is proved that the uniform  convergence of a sequence of equi-semi-concave functions implies their convergence $C^1$ in some sense. This implies for the function $u$ given in Theorem \ref{Tgenecont} a little more regularity. More precisely, if $c_n\rightarrow c$, if $\theta_n \rightarrow \theta$ and if $u_{c_n}$ is derivable at $\theta_n$ and $u_c$ at $\theta$, we have
 $$\lim_{n\rightarrow \infty} \frac{\partial u}{\partial \theta}(\theta_n, c_n)=\frac{\partial u}{\partial \theta}(\theta, c).$$
In fact, using technics given in \cite{Arna4},  we will prove in Appendix \ref{AppB3} that  the map that maps $c$ on the   full pseudograph\footnote{see the definition in subsection \ref{ssweakk}}  $\mathcal{PG}(c+  u_c') =\{ (0,c)+\partial  u_c (t), \quad t\in \T\}$ of $c+u_c'$ is continuous for the Hausdorff distance.  \end{enumerate}
 \end{remks}
 
\begin{theorem}\label{Tpseudofol}
 With the notations of Theorem \ref{Tgenecont}, we have
 \begin{enumerate}
 \item $\displaystyle { \bigcup_{c\in \R} \mathcal{PG} (c+u'_c)= \A;}$
 \item if the rotation number\footnote{see point \ref{PtAubrypseudo} in subsection \ref{ssweakk}.}  associated to $c$ is strictly smaller than the  one associated to $c'$, then for all $(q,p)\in \mathcal{PG} (c+u'_c)$ and $(q, p')\in \mathcal{PG} (c+u'_{c'})$, we have $p<p'$.
 \end{enumerate} 
 \end{theorem}
  \begin{remks}\begin{enumerate}
\item    In \cite{KatOrn} the authors prove the existence of such a family of pseudographs.  Here we provide a precise example of such a family of pseudographs.
\item As a result of the proof, we will deduce (see Proposition \ref{Pordreirrat}) that the Aubry-Mather sets are contained in pseudographs that are vertically ordered by their rotation numbers. 
\end{enumerate}
 \end{remks}

Hence there always exists a continuous choice of  weak K.A.M. solutions. We can ask when $u$ is more regular. To answer the question, let us recall that a twist map is said to be $C^0$-integrable if the annulus $\T\times \R$ is $C^0$-foliated by $C^0$ invariant graphs. We obtain the following result.
\begin{theorem}\label{Tgeneder}
With the notations of Theorem \ref{Tgenecont}, we have equivalence of 
\begin{enumerate}
\item\label{pt1Tgeneder} $f$ is $C^0$-integrable; 
\item\label{pt2Tgeneder}  the map $u$ is  $C^1$. 
\end{enumerate} Moreover, in this case, $u$ is unique and  we have\footnote{See the notation $\pi_1$ at the beginning of subsection \ref{ssnota}.}
 \begin{itemize}
 \item the graph of $c+u_c'$ is a leaf of the invariant foliation;
 \item $h_c:\theta \mapsto \theta+\frac{\partial u}{\partial c}(\theta,c)$ is a semi-conjugation between the projected Dynamics $g_c: \theta\mapsto \pi_1\circ f\big(\theta, c+\frac{\partial u}{\partial \theta}(\theta,c)\big)$ and a rotation $R$ of $\T$, i.e. $h_{c}\circ g_c=R\circ h_{c}.$
   \end{itemize}
 \end{theorem}
 We give in Appendix \ref{sscomplint} an example of a $C^\infty$ integrable symplectic twist map (coming from an autonomous Tonelli Hamiltonian)   for which the discounted method doesn't select a transversely continuous weak K.A.M. solution. \\

 \begin{remks}\begin{enumerate}
\item What is the most surprising  in Theorem \ref{Tgeneder} is the fact that in the $C^0$-integrable case the semi-conjugation $h_c$ continuously depends on $c$ even at a $c$ where the rotation number is rational. At an irrational rotation number, this is an easy consequence of the uniqueness of the invariant measure supported on the corresponding leaf. What happens for a rational rotation number is more subtle.
\item Observe that   in the $C^k$-integrable case for $k\geq 1$, we can only claim that $u$ and $\frac{\partial u_c}{\partial \theta}$ are $C^k$; so in the $C^0$ case, even the derivability with respect to $c$ is surprising; this surely is related to the 2-dimensional setting in which we work.
\item The function   $u$ can be seen as  a $C^1$ generating function of a  continuous map $H:\T\times \R\rightarrow \T\times \R$ that is defined by
 \begin{equation}\label{Esemiconj}H(\theta ,r)=(x, c)\Longleftrightarrow x= \theta+\frac{\partial u}{\partial c}(\theta, c)  \quad{\rm and} \quad c=  r-\frac{\partial u}{\partial \theta}(\theta,c)
 .\end{equation}
 and that satisfies
 $H\circ f(\theta, r)=H(\theta, r)+(\rho(c), 0)$ with $\rho:\R\rightarrow\R$ increasing homeomorphism. Observe that if on some curve of the invariant foliation, the Dynamics is not recurrent (i.e. we have a Denjoy counter-example), then $H$ is not an homeomorphism because it is not injective. 
\item In \cite{Eva}, the author makes similar remarks concerning the link between the weak K.A.M. solutions and a generating function in the  Hamiltonian case.
\item In the proof we also see that the foliation $\eta_c=c+\frac{\partial u_c}{\partial \theta}$ has a partial derivative with respect to $c$ along any leaf having a rational rotation number.
 \item The results that we obtain in Theorems \ref{Tgeneder} and \ref{corLip}  concerning $C^0$ and Lipschitz integrability are in fact local and we could give similar results on a bounded annulus.
\end{enumerate}\end{remks}

An interesting question concerns the restricted Dynamics to the leaves in the $C^0$-integrable case. A priori, such a Dynamics can be a Denjoy counter-example (but we have no example for such a phenomenon). With more regularity of the foliation, we obtain  a more precise  result. We need some definitions.
\begin{defi}
\begin{itemize}
\item A {\em symplectic homeomorphism} is a homeomorphism that is locally a $C^0$ uniform limit of symplectic diffeomorphisms;
\item an {\em exact symplectic homeomorphism} is a homeomorphism that is locally a $C^0$ uniform limit of exact symplectic diffeomorphisms, where we recall that a diffeomorphism $f:\A\rightarrow \A$ is exact symplectic if $f$ is homotopic to Id and  the 1-form $f^*(rd\theta)-rd\theta$ is exact;
\item if $f:\A\rightarrow\A$ is a symplectic homeomorphism, {\em $C^0$ Arnol'd-Liouville coordinates} are given by a symplectic homeomorphism $\Phi:\A\rightarrow \A$ such that the standard foliation into graphs $\T\times \{ c\}$ is invariant by $\Phi\circ f\circ \Phi^{-1}$ and $\Phi\circ f\circ \Phi^{-1}(x, c)=(x+\rho(c), c)$ for some (continuous) function $\rho:\R\rightarrow \R$.

\end{itemize} 
\end{defi}
\begin{remks}
\begin{enumerate}
\item An orientation preserving homeomeorphism is symplectic if and only if it preserves the Lebesgue measure;
\item observe that if $f$ is exact symplectic and maps the graph of $c+u'$ onto the graph of $c'+v'$ where $c, c'\in\R$ and $u, v:\T\rightarrow \R$ are $C^1$, then $c=c'$;
\item if a symplectic homeomorphism $\phi:\A\rightarrow \A$ gives some Arnol'd-Liouville coordinates for $f$, then, composing  $\phi$ with $(x, c)\mapsto (\pm x, \pm c +c_0)$ for some $c_0\in \R$, we may assume that $\phi$ is exact symplectic.
\end{enumerate}
\end{remks}

\begin{theorem}\label{corLip} 
With the notations of Theorem \ref{Tgeneder}, we have equivalence of 
\begin{enumerate}
\item\label{pt1corLip} $f$ is Lipschitz integrable\footnote{See the definition in subsection \ref{ss31}}; 
\item\label{pt2corLip}  the map $u$ is $C^1$ with
\begin{itemize}
\item  $\frac{\partial u}{\partial \theta}$ locally Lipschitz continuous; 
\item$\frac{\partial u}{\partial c}$ uniformly Lipschitz continuous in the variable  $\theta$ on any compact set of $c$'s;
\item  for every compact subset $\ck\subset \A$, there exists a constant $k>-1$ such that $\frac{\partial^2 u}{\partial \theta\partial c}>k$  Lebesgue almost everywhere in $\ck$.
\end{itemize} 
\end{enumerate}
  In this case,  there exists $\Phi:\T\times\R\rightarrow \T\times \R$  exact symplectic homeomorphism that is $C^1$ in the $\theta$ variable  and maps the invariant foliation onto the standard one such that:
$$\forall (x, c)\in \T\times\R,\quad  \Phi\circ f\circ \Phi^{-1}(x, c)=(x+\rho(c), c);$$
where $\rho:\R\rightarrow \R$ is an increasing homeomorphism.
 
\end{theorem}
\begin{remks} \begin{enumerate} 
\item  In this case, all the leaves are $C^1$ and the foliation is a $C^1$ lamination\footnote{See the definition at the beginning of subsection \ref{ss31}.}; moreover the Dynamics restricted to every leaf is $C^1$ conjugate to a rotation;
\item   the last part of Theorem \ref{corLip}  provides some  Arnol'd-Liouville coordinates. 
\end{enumerate}
\end{remks}
 Let us give two statements that will bring some light on these Arnol'd-Liouville coordinates. The first one asserts that a symplectic homeomorphism that has an invariant foliation that is exact symplectically homeomorphic to the standard one admits some global Arnol'd-Liouville coordinates.
\begin{propos}\label{Psymfolarn}
 Let $f:\A \to \A$ be an exact symplectic homeomorphism.  Let us assume that $f$ preserves { each leaf of} a foliation $\cF$ into $C^0$ graphs  which is  symplectically homeomorphic (by $\Phi : \A \to \A$) to the standard foliation $\cF_0=\Phi(\cF)$. Then there exists a continuous function $\rho:\R\rightarrow \R$ such that
 $$\forall (\theta, r)\in \A, \quad \Phi\circ f\circ \Phi^{-1}(\theta, r)=(\theta+\rho(r), r).$$
 \end{propos}
{ 
\begin{remks}
The same proof applies in the slightly more general case where $f$ preserves the foliation $\cF$, possibly sending a leaf on a different one. Then the conclusion should be modified by: 
$$\exists r_0\in \R, \exists \lambda \in \{-1,1\},\forall (\theta, r)\in \A, \quad \Phi\circ f\circ \Phi^{-1}(\theta, r)=(\theta+\rho(r), \lambda r+r_0).$$
\end{remks}

}

Next, we characterize foliations by continuous  graphs that are exact symplectically homeomorphic to the standard foliation. The reader will notice a strong resemblance  with Theorems \ref{Tgeneder} and \ref{corLip}.

\begin{theorem}\label{TC0arn}
A $C^0$-foliation of $\A$: $(\theta, c) \mapsto \big(\theta , \eta_c(\theta)\big)$, where $\int_\T \eta_c = c$,  is exact symplectically homeomorphic to the standard foliation if an only if there exists a $C^1$ map $u : \A \to \R$ such that 
\begin{itemize}
\item $u(0,r) = 0$ for all $r\in \R$,
\item $\eta_c(\theta) = c+ \frac{\partial u}{\partial \theta}(\theta,c)$ for all $(\theta, c) \in \A$,
\item for all $c\in \R$, the map $\theta \mapsto \theta + \frac{\partial u}{\partial c}(\theta , c)$ is a homeomorphism of $\T$.
\end{itemize}

\end{theorem}
Using this characterization, we will be able to give (see  section \ref{Sfollag}) an example of a $C^0$-foliation of $\A$ into smooth graphs that is not exact symplectically homeomorphic to the standard foliation.

\begin{exa}
Let $\varepsilon:\R\rightarrow \R$ be a non-$C^1$ function that is $\frac{1}{4\pi}$-Lipschitz. Then the function
$$(\theta,c)\mapsto u_c(\theta)=\frac{\varepsilon(c)}{2\pi}\sin(2\pi \theta)$$
defines a Lipschitz foliation of $\A$ into smooth graphs of $\theta\in\T\mapsto c+\varepsilon(c)\cos (2\pi \theta)$ that is not symplectically homeomorphic to the standard foliation.\\
A result of Theorem \ref{corLip}  is that this foliation cannot be invariant by a symplectic twist map.
\end{exa}
\begin{remk} Theorem \ref{TC0arn} seems to be global. In fact, one can provide an analogous local statement and an example of a local continuous foliation in $C^0$ graphs that is not  straightenable by a (local) symplectic homeomorphism.
\end{remk}
\begin{cor}\label{Corota}
A symplectic twist map $f: \A \to \A$ is $C^0$-integrable with the dynamics on each leaf conjugated to a rotation if and only if it admits global $C^0$ Arnol'd-Liouville coordinates.
\end{cor}

To prove all these results, we will use together Aubry-Mather theory, weak K.A.M. theory in the discrete case and also   ergodic theory. Let us detail what will be in the different sections
\begin{itemize}
\item Section \ref{secth11} contains some reminders on twist maps,  Aubry-Mather theory, on discrete weak K.A.M. theory, some new results on the weak K.A.M. solutions and the proof of   Theorems \ref{Tgenecont}  and \ref{Tpseudofol};
\item Section \ref{secfirst} contains the proof of the first implication of Theorem \ref{Tgeneder}; after recalling some generalities about twist maps, we consider the case of the rational curves by using some ergodic theory, then we build the wanted function $u_c$ and prove its regularity by using also ergodic theory;
\item the second implication of Theorem \ref{Tgeneder} is proved in section \ref{secsecond};
\item Theorem \ref{corLip}  is proved in section \ref{secth13};
 \item Section \ref{Sfollag} contains proofs of Proposition \ref{Psymfolarn} and Theorem \ref{TC0arn} and gives an example of a $C^0$ foliation of $\A$ into continuous graphs that is not exact symplectically homeomorphic to the standard foliation.
 \item Appendix \ref{AppA} contains some examples, Appendix \ref{Apfulpseudo} deals with full pseudographs and Appendix  \ref{ssGreenb} recalls some results about Green bundles.
\end{itemize}

\subsection*{Acknowledgements} The authors are grateful to Philippe Bolle and Fr\' ed\' eric Le Roux for insightful discussions that helped clarify and simplify some proofs of this work.

 \section{ Aubry-Mather and weak K.A.M. theories for twist maps and proof of Theorems \ref{Tgenecont} and \ref{Tpseudofol}}\label{secth11}
 \subsection{The setting}\label{ssnota}  The definitions and results that we give here are very classical now. Good references are \cite{ForMat}, \cite{Gol1}, \cite{Mat2}, \cite{Mos}, \cite{Be2}.
 
 Let us introduce some notations
 \begin{notas}
 \begin{itemize}
 \item $\T=\R/\Z$ is the circle and $\A=\T\times \R$ is the annulus ;  $\pi: 
 \R\rightarrow \T$ is the usual projection;
 \item the universal covering of the annulus is denoted by $p:\R^2\rightarrow \A$;
 \item the corresponding projections are $\pi_1: (\theta, r)\in \A\mapsto \theta\in \T$ and $\pi_2: (\theta, r)\in \A\mapsto r\in \R$; we denote also the corresponding projections of the universal covering by $\pi_1$, $\pi_2~: \R^2\rightarrow \R$;
 \item the Liouville 1-form is defined on $\A$ as being $\lambda=\pi_2d\pi_1=rd\theta$; then $\A$ is endowed with the symplectic form $\omega=-d\lambda$.
 \end{itemize}
\end{notas}
 We gave in the introduction the definition of a symplectic twist map.
A $C^2$ generating function $S:\R\times\R\rightarrow \R$   that satisfies the following definition can be associated  to any lift $F$ of such a symplectic twist map $f$.
 
\begin{defi}
The $C^2$ function $S:\R^2\rightarrow \R$ is a {\em generating function} of the lift $F:\R^2\rightarrow \R^2$ of a symplectic twist 
map if
\begin{itemize}
\item $S(\theta+1, \Theta+1)=S(\theta, \Theta)$;
\item $\displaystyle{\lim_{|\Theta-\theta|\rightarrow \infty}\frac{S(\theta, \Theta)}{|\Theta-\theta|}=+\infty}$; we say that $S$ is {\em superlinear}; 
\item for every $\theta_0, \Theta_0\in \R$, the maps $\theta\mapsto \frac{\partial S}{\partial \Theta}(\theta, \Theta_0)$ and $\Theta\mapsto \frac{\partial S}{\partial \theta}(\theta_0, \Theta)$ are decreasing diffeomorphisms of $\R$;
\item for $(\theta, r), (\Theta, R)\in \R^2$, we have the following  equivalence $$F(\theta, r)=(\Theta, R)\Leftrightarrow r=-\frac{\partial S}{\partial \theta}(\theta, \Theta)\quad{\rm and}\quad R=\frac{\partial S}{\partial \Theta}(\theta, \Theta).$$
\end{itemize}
\end{defi}

 \begin{remk} 
 J.~Moser proved in \cite{Mos} that such a twist map is the time 1 map of a $C^2$ 1-periodic in time Hamiltonian $H:\T\times\R\times\R\rightarrow \R$  that is $C^2$ convex in the  fiber direction\footnote{In fact J.~Moser assumed that $f$ is smooth.}, i.e. such that $$\frac{\partial^2 H}{\partial r^2}(\theta, r, t)>0.$$
 Then there exists a relation between the Hamiltonian that was built by J.~Moser and the generating function. Indeed, if we denote by $(\Phi_t)$ the time $t$ map of the Hamiltonian $H$ that is defined on $\R^2$ and by $L$ the associated Lagrangian that is defined by
 $$L(\theta, v, t)=\max_{r\in\R}\big(rv-H(\theta, r, t)\big),$$
then we have
 \begin{itemize}
 \item for every $t\in (0, 1]$, $\Phi_t$ is a symplectic twist map and $\Phi_1=F$;
 \item there exists a $C^1$ time-dependent family of $C^2$ generating functions $S_t$ of $\Phi_t$ such $S_1=S$ and for all $ (\theta, r), (\Theta, R)\in\R^2,$
 $$ \Phi_t(\theta,r)=(\Theta, R)\Rightarrow S_t(\theta, \Theta)=\int_0^tL\big(\pi_1\circ \Phi_s(\theta, r), \frac{\partial}{\partial s}\big( \pi_1\circ \Phi_s(\theta, r)\big),s\big)ds.$$
 \end{itemize}
 In other words, the generating function is also the Lagrangian action.
 \end{remk}
 \subsection{Aubry-Mather theory} Good references for what is in this section are \cite{Ban}, \cite{Gol1} and \cite{Arna3}.
 Let us recall the definition of some particular invariant sets.
 \begin{defi} Let $F:\R^2\rightarrow \R^2$ be a lift of a symplectic twist map $f$.
 \begin{itemize}
 \item a subset $E$ of $\R^2$ is {\em well-ordered} if it is invariant under the translation $(\theta, r)\mapsto (\theta+1, r)$ and $F$ and if for every $x_1, x_2\in E$, we have
 $$\big[ \pi_1(x_1)<\pi_ 1(x_2)\big]\Rightarrow \big[ \pi_1\circ F(x_1)<\pi_1\circ F(x_2)\big];$$
 this notion is independent from the lift of $f$ we use;
 \item a subset $E$ of $\A$ is {\em well-ordered} if $p^{-1}(E)$ is well-ordered;
 \item an {\em Aubry-Mather set } for $f$ is a compact well-ordered set or the lift of such a set;
  \item a piece of orbit $(\theta_k, r_r)_{k\in [a, b]}$ for $F$ is {\em minimizing} if for every sequence  $ (\theta'_k)_{k\in [a, b]}$ with $\theta_a=\theta'_a$ and $\theta_b=\theta'_b$, it holds
 $$\sum_{j=a}^{b-1}S(\theta_j, \theta_{j+1})\leq \sum_{j=a}^{b-1}S(\theta'_j, \theta'_{j+1});$$
 then we say that $(\theta_j)_{j\in [a, b]}$ is a {\em minimizing sequence} or {\em segment};
 \item {  an infinite piece of orbit, or a full orbit for $F$ is minimizing if all its finite subsegments are minimizing;}
 \item an invariant set is said to be minimizing if all the orbits it contains are minimizing.
 \end{itemize}
 \end{defi}
 The following properties of the well-ordered sets are well-known
 \begin{enumerate}
 \item a minimizing orbit and its translated orbits by $(\theta, r)\mapsto (\theta+1, r)$  define a well-ordered set;
 \item the closure of a well-ordered set is a well-ordered set;
 \item\label{Ptcirclehomeom} any well-ordered set $E$ is contained in the (non-invariant) graph of a Lipschitz map $\eta:\T\rightarrow \R$;  it follows that the map $N = \big(\cdot , \eta(\cdot)\big) : \T \to {\rm Graph}(\eta)$ is Lipschitz and so are the maps  $\pi_1\circ f\circ N_{ |\pi_1(E)}$ and $\pi_1\circ f^{-1}\circ N_{ |\pi_1(E)}$ . This  implies that the projected restricted Dynamics $\pi_1\circ f\big(\cdot, \eta(\cdot)\big)_{|\pi_1(E)}$ to an Aubry-Mather set is the restriction of a biLipschitz orientation  preserving circle homeomorphism;
 \item any well-ordered set $E$ in $\R^2$  has a unique rotation number $\rho(E)$ \big(the one of the circle homeomorphism we mentioned in  Point (\ref{Ptcirclehomeom})\big), i.e.
 $$\forall x\in E, \quad \lim_{k\rightarrow\pm\infty} \frac{1}{k}\big(\pi_1\circ F^k(x)-\pi_1(x)\big)=\rho(E);$$
 \item for every $\alpha\in \R$, there exists a minimizing Aubry-Mather set $E$  such that $\rho (E)=\alpha$; 
 \item  if $\alpha$ is irrational, there is a unique  minimizing Aubry-Mather that is minimal (resp. maximal) for the inclusion; the minimal one is then a Cantor set  or a complete graph and the maximal one $\cm(\alpha)$ is the union of the minimal one and orbits that are homoclinic to the minimal one;
 \item if $\alpha$ is rational, any Aubry-Mather set that is minimal for the inclusion is a periodic orbit;
 \item\label{ptgraphminimizing}  any essential invariant curve by a symplectic twist map is in fact a Lipschitz graph (Birkhoff theorem, see \cite {Bir1}, \cite{Fa1}  and \cite {He1}) and a well-ordered set.
   \end{enumerate}
 We will need more precise properties for minimizing orbits.  \begin{defi}{\rm Let $a=(a_k)_{k\in I}$ and $b=(b_k)_{k\in I}$ be two finite or infinite sequences of real numbers. Then
 \begin{itemize}
 \item if $k\in I$, we say that $a$ and $b$ cross at $k$ if $a_k=b_k$;
 \item if $k, k+1\in I$, we say that $a$ and $b$ cross between $k$ and $k+1$ if $(a_k-b_k)(a_{k+1}-b_{k+1})<0$.
 \end{itemize}

}
\end{defi}
{Note that concerning the first item, the traditional terminology also imposes that $(a_{k-1}-b_{k-1})(a_{k+1}-b_{k+1})<0$ when $k$ is in the interior of $I$. However, due to the twist condition, this is automatic  for projections of orbits of $F$  as soon as $a_k=b_k$  if the two orbits are distinct.}

\begin{propos}\label{PAubryfund}{\bf (Aubry fundamental lemma)} Two distinct minimizing sequences cross at most once except possibly at the two endpoints  when the sequence is finite.

\end{propos}
\subsection{Classical results on weak K.A.M. solutions}\label{ssweakk}
Good references are \cite{Be1},  \cite{Be2} or \cite{GT1}. 
We assume that $S$ is a generating function of a lift $F$ of a symplectic twist map $f$.

We define on $C^0(\T, \R)$ the so-called {\em negative Lax-Oleinik maps} $T^c$ for $c\in \R$ as follows:

 if $ u\in C^0(\T, \R)$, we denote by $\tilde u : \R \to \R$ its lift and  
\begin{equation} \label{Edefweakk}  \forall \theta\in \R, \quad    \widetilde T^c \tilde u(\theta)=\inf_{\theta'\in\R} \big(\tilde u(\theta')+S(\theta', \theta)+c(\theta'-\theta)\big).\end{equation}
The function $ \widetilde T^c \tilde u$ is then  $1$-periodic and the negative  Lax-Oleinik operator is defined as the induced map $T^c u : \T \to \R$.

An alternative but equivalent definition is as follows (see also \cite{Za} for similar constructions): define the function 
\begin{equation}\label{projS}
\forall (x,x') \in \T\times \T, \quad S^c (x,x') = \inf_{\substack{\pi(\theta) = x \\ \pi(\theta')=x'}}  S(\theta, \theta')+c(\theta-\theta').
\end{equation}
Then
$$\forall x\in \T,\quad T^c u(x) = \inf_{x'\in \T} u(x')+S^c(x',x).$$

Then it can be proved that there exists a unique function $\alpha: \R\rightarrow \R$ such that the map $\widehat T^c=T^c+\alpha(c)$ that is defined by
$$\widehat T^c(u)=T^c(u)+\alpha(c)$$
has at least one fixed point in $C^0(\T, \R)$, i.e.  if $u\in C^0(\T, \R)$ is such a fixed point, its lift verifies 
\begin{equation} \label{Esolweakk}  \forall \theta\in \R, \quad   \tilde u(\theta) = \inf_{\theta'\in\R} \big(\tilde u(\theta')+\widetilde S(\theta', \theta)+c(\theta'-\theta)+\alpha(c)\big).\end{equation}
 
Such a fixed point is called a {\em weak K.A.M. solution}. It is not necessarily unique.  For example,  if $u$ is a weak K.A.M. solution, so is $ u+k$ for every $k\in \R$, but there can also be other solutions. We denote by $\cs_c$ the set of these weak K.A.M. solutions.  There is no link in general for solutions corresponding to distinct $c$'s.  We recall
\begin{defi}
Let $u:\R\rightarrow\R$ be a function and let $K>0$ be a constant. Then $u$ is $K$-semi-concave if for every $x$ in $\R$, there exists some $p\in \R$ so that:
$$\forall y\in \R,\quad u(y)-u(x)-p(y-x)\leq \frac{K}{2}(y-x)^2.$$

 A function $v :\T\to \R$ is $K$-semi-concave if  its lift $\tilde v: \R \to \R$ is.
\end{defi}
A good reference for semi-concave functions is the appendix A of  \cite{Be2} or \cite{cansin}.
\begin{nota}
If $u\in C^0(\T, \R)$ and $c\in \R$, we will denote by $\cg(c+u')$ the partial graph of $c+u'$. This is a graph above the set of derivability of $u$.\\
When $u$ is semi-concave, we sometimes say that $\cg(c+u')$ is a {\em pseudograph}.
\end{nota}

Let us end  with  definitions:
 \begin{defi}\label{Clarke}
 Let $g:\T \to \R$ be a Lipschitz function (hence derivable almost everywhere). We define 
 $$\forall x\in \T, \quad \partial g(x) = {\rm co} \big\{ (x,p)\in  \T\times\R, \ \ (x,p)\in \overline{ \cg (g')}\big\}.$$
 \end{defi}
The notation co stands for the convex hull  in the fiber direction. The sets  $\partial g(x)$ are non empty, (obviously) convex and compact. They are particular instances of the Clarke subdifferential. This set is a good candidate for a generalized derivative because if $g$ is derivable at $x$ then $g'(x)\in \partial g(x)$. Moreover, if $\partial g(x)$ is a singleton, then $g$ is derivable at $x$. The converse is in general not true, but it is however true for semi-concave functions. 

\begin{defi}
If $g : \T \to \R$ is Lipschitz and $c\in \R$, we define  $\mathcal{PG}(c+g') =\{ (0,c)+\partial g (t), \quad t\in \T\}$.  If $g$ is semi-concave, we call it the full pseudograph of $c+g'$.
\end{defi}

 A proof of the  following proposition is given in Appendix \ref{Apfulpseudo}.
\begin{propos}\label{hausdorff}
Let $(f_n)_{n\in \mathbb N} $ be a sequence of equi-semi-concave functions from $\T$ to $\R$ that converges (uniformly) to a function $f$ (that is hence also semi-concave).

Then $\mathcal{PG}(f'_n) $ converges to $\mathcal{PG}(f')$ for the Hausdorff distance.
\end{propos}

The following results can be found in the papers that we quoted

 \begin{enumerate}[label=(\alph*)]
\item\label{first} the function $\alpha$ is convex   and superlinear;
\item if $u\in C^0(\T, \R)$, then $\widehat T^cu$  is semi-concave and then differentiable Lebesgue almost everywhere; 
\item\label{ptdifftt} the function  $\widehat T_c u$ is differentiable at $x$ if and only if there is only one $y$ where the minimum is attained in Equality (\ref{Esolweakk}); in this case, if $u$ is semi-concave, then it is differentiable at $y$ and we have 
$$f\big(y, c+u'(y)\big)=\big(x, c+(\widehat T^cu)'(x)\big);$$
if $u$ is a weak K.A.M. solution for $\widehat T^c$ {that is differentiable at $x$ then} $\Big(f^{k}\big(x, c+u'(x)\big)\Big)_{k\in\Z_-}$ is a minimizing piece of orbit that is contained in $\cg(c+u')$;
\item\label{same} moreover, for any compact subset $K$ of $\R$, the weak K.A.M. solutions for $T^c$ with $c\in K$ are uniformly semi-concave (i.e. for a fixed constant of semi-concavity) and then uniformly Lipschitz;
\item\label{ptinvgraph} if $u\in C^0(\T, \R)$, then 
$$f^{-1}\big(\overline{\cg(c+(\widehat T^cu)')}\big)\subset \cg(c+u');$$
if $u$ is a weak K.A.M. solution for $\widehat T^c$, then $\cg(c+u')$ satisfies
$$f^{-1}\big(\overline{\cg(c+u')}\big)\subset \cg(c+u')$$
and for every $(\theta,r)\in \overline{\cg(c+u')}$, then $\big(f^{k}(\theta,r)\big)_{k\in\Z_-}$ is minimizing;\\
we will give in Appendix \ref{sspseudo} an example of a backward invariant pseudograph that doesn't correspond to any weak K.A.M. solution;
\item \label{PtAubrypseudo} moreover, if $u$ is a weak K.A.M. solution  for $\widehat T^c$, then the set $$\bigcap_{n\in\N}f^{-n}\big(\cg(c+u')\big)$$ is a { $f$-invariant} minimizing compact well-ordered set to which we can associate a unique rotation number. It results from Mather's theory that this rotation number only depends on $c$ and is equal to $\rho(c)=\alpha'(c)$; because of the convexity of $\alpha$, observe in particular that $\alpha$ is $C^1$ and $\rho$ is continuous and non-decreasing;   

\item it then follows from the first \ref{first} and the previous  \ref{same} and \ref{PtAubrypseudo} points that, as in \ref{same}, for any compact subset $K$ of $\R$, the weak K.A.M. solutions for $T^c$ with $\rho(c)\in K$ are uniformly semi-concave (i.e. for a fixed constant of semi-concavity) and then uniformly Lipschitz;

\item in the setting of point \ref{PtAubrypseudo}, then for every weak K.A.M. solution for $\widehat T^c$, the graph $\cg(c+u')$ contains any minimizing Aubry-Mather set with rotation number $\rho(c)$ that is minimal for the inclusion; we denote the union of these  Aubry sets by $\ca\big(\rho(c)\big)$. If $\rho(c)$ is irrational,  then two possibilities may occur:
\begin{itemize}
\item either $\ca\big(\rho(c)\big)$ is  an invariant  Cantor set  and  $\cg(c+u')$ is contained in the unstable set of the Cantor set $\ca\big(\rho(c)\big)$;
\item  or $\ca\big(\rho(c)\big)=\cg(c+u')$ and $u$ is $C^1$.
\end{itemize}
 If $\rho(c)$ is rational, then $\ca\big(\rho(c)\big)$ is the union of some periodic orbits   and  $\cg(c+u')$ is contained in the union of the unstable sets of these periodic orbits.

\end{enumerate}
We noticed that to any $c\in \R$ there corresponds a unique rotation number $\rho(c)$. But it can happen that distinct numbers $c$ correspond to a same rotation number $R$. In this case, because $\rho(c)=\alpha'(c)$ is non decreasing (because of  point \ref{PtAubrypseudo}), $\rho^{-1}(R)=[c_1, c_2]$ is an interval. It can be proved that this may happen  only for rational $R$'s. 
This is a result of John Mather. A simple proof can be found in \cite[Proposition 6.5]{Be3}.

Finally, when $c$ corresponds to an irrational rotation number $\rho(c)$, then there exists only one weak K.A.M. solution up to constants.  
 The argument comes from \cite{Be3}.

\begin{nota}
When $\rho(c)$ is irrational, we will denote by $u_c$ the  (unique) solution such that $u_c(0)=0$.\end{nota}
\subsection{More results on weak K.A.M. solutions}\label{ssmoreK.A.M.}

 We start with a lemma stating that some minimizing sequences admit a rotation number:

\begin{lemma}
Let $u$ be a weak K.A.M. solution  for $\widehat T^c$. Let $(\theta_0,r)\in \overline{\cg(c+u')}$, and $\tilde \theta_0\in \R$ a lift of $\theta_0$. Let  $(\tilde\theta_k,r_k)_{k\in\Z_-}= \big(F^{k}(\tilde\theta,r)\big)_{k\in\Z_-}$. Then
$$\lim_{k\to -\infty} \frac{\tilde\theta_k}{k} = \rho(c).$$
\end{lemma}
\begin{proof}
Let us argue by contradiction. If this is not the case, there exists $\varepsilon>0$ and a subsequence $n_k\to -\infty$ such that for all $k$, 
$$ \left| \frac{\tilde\theta_{n_k}}{n_k} - \rho(c)\right|  >\varepsilon.$$ 

Up to an extra extraction, we may assume that the following sequence of measures converges:
$$\lim_{k\to -\infty}\frac{1}{|n_k|} \sum_{i=-n_k}^{-1}\delta_{(\tilde\theta_i,r_i)} = \mu.$$
We then know that $\mu$ is a minimizing Mather measure whose support is made of points having  rotation number  $\rho(c)$. Consider the function $D(\tilde\theta,r) = \pi_1\circ F(\tilde\theta,r) - \tilde \theta$. It is a periodic function in $\tilde\theta$ that is the lift of a function on $\A$.

We may compute that 
\begin{multline*}
\left| \int D d\mu - \rho(c)\right| = \left| \lim_{k\to-\infty} \frac{1}{|n_k|} \sum_{i=-n_k}^{-1}D(\tilde\theta_i,r_i) - \rho(c)\right| \\
= \left| \lim_{k\to-\infty} \frac{1}{|n_k|} (\tilde\theta_0 - \tilde\theta_{n_k}) - \rho(c)\right| \geq \varepsilon.
\end{multline*}
This contradicts the fact that $\mu$ has rotation number $\rho(c)$.
\end{proof}

\begin{propos}\label{Pordreirrat}
Let $u_1$, $u_2$ be two weak K.A.M. solutions corresponding to $T^{c_1}$, $T^{c_2}$, such that $\rho(c_1)<\rho(c_2)$. Then we have
\begin{itemize}
\item $c_1<c_2$;
\item { for any $t\in \T$, if $(t,p_1)\in \partial u_1(t)$  and   $(t,p_2)\in \partial u_2(t)$ then $c_1+p_1 < c_2+p_2$; }
\item in particular, at every point of differentiability $t$ of $u_1$ and $u_2$: $c_1+u'_{1}(t)<c_2+u'_{2}(t)$.
\end{itemize}

\end{propos}
\begin{proof}
Let $\tilde u_1$ and $\tilde u_2$ be the lifts of $u_1$ and $u_2$. We introduce the notation $v(t)=\tilde u_{2}(t)-\tilde u_{1}(t) +(c_2-c_1)t$. Then $v$ is Lipschitz and thus Lebesgue everywhere differentiable and equal to a primitive of its derivative. 
Let us assume by contradiction  that  there exist $(x, c_1+p_1)\in \overline{ \cg(c_1+u_1')}$  and   $(x, c_2+p_2)\in \overline{\cg (c_2+u_2')}$
\begin{equation}\label{EineqAubry} c_2+p_2 \leq c_1+p_1.\end{equation}
As $\rho(c_1)\not=\rho(c_2)$,  the two graphs correspond to distinct rotation numbers. Thanks to \ref{ptinvgraph} their closures have no intersections. The inequality (\ref{EineqAubry}) is then strict. \\
We introduce the notation 
$(x^1, y^1)=(x, c_1+p_1)$ and $(x^2, y^2)=(x, c_2+p_2)$.
 Then the orbit of $(x^i, y^i)$ is denoted by $(x^i_k, y^i_k)_{k\in\Z}$. We know that the negative orbits $(x^i_k, y^i_k)_{k\in\Z_-}$, that are contained in the corresponding graphs, are minimizing. Hence the sequences $(x^i_k)_{k\in\Z_-}$ are minimizing. By Aubry's fundamental lemma, we know that they can cross at most once (hence only at $x$). But we have
\begin{itemize}
\item because of the twist condition, as $x^1_0=x^2_0$ and $y^1_0>y^2_0$, then $x^1_{-1}<x^2_{-1}$;
\item as  $\rho(c_1)<\rho(c_2)$,
and thus for $k$ close enough to $-\infty$, we have: $x^1_k>x^2_k$.
\end{itemize}
Finally we find two crossings for two minimizing sequences, a contradiction.

 We have in particular for any point $t$ of derivability of $u_1$ and $u_2$
$$c_1+u_1'(t)<c_2+u_2'(t).$$
Integrating this inequality, we deduce that $c_1<c_2$.

Finally, for any $t\in \T$, 
as for all
$(t,p_1)\in \overline{ \cg(c_1+u_1')}$  and   $(t,p_2)\in \overline{\cg (c_2+u_2')}$
\begin{equation}\label{RineqAubry} c_2+p_2 > c_1+p_1,\end{equation}
taking convex hulls, we get the result.
 
\end{proof}
As to an  irrational rotation  number a unique constant $c$ corresponds, we deduce the following corollary.
\begin{cor}\label{Cordre}
With the same notation as in Proposition \ref{Pordreirrat}, assume that $c_1<c_2$ are such that at least one of $\rho(c_1)$ and $\rho(c_2)$ is irrational. Then
the function $t\in\R\mapsto\tilde  u_{c_2}(t)-\tilde u_{c_1}(t) +t(c_2-c_1)$ is strictly increasing.
\end{cor}
\begin{remk}
A consequence of Proposition \ref{Pordreirrat} is that the { pseudo}graphs corresponding to the weak K.A.M. solutions having distinct  rotation numbers are vertically ordered with the same order as the one between the rotation numbers.\\
\end{remk} 
Now we  recall some results that are contained in \cite{GT1} (see especially Theorem 9.3).

\begin{nota}
If $\theta_1, \theta_2$ are in $\R$, $  c\in\R$  and $n\geq 1$, we denote by $  \cs^c_n:\R\times \R\rightarrow \R$ the function that is defined by
$$ \cs^c_n(\theta, \Theta)=\inf_{\substack{\theta_0=\theta \\ \theta_n-\Theta\in\Z}} \left\{\sum_{i=1}^n\big( S(\theta_{i-1},\theta_i)+c(\theta_{i-1}-\theta_i)\big)\right\}.$$
 Observe that $\cs^c_n$ is $\mathbb Z^2$-periodic.
 \end{nota}

\begin{enumerate}
\item\label{Matherinv} If $R$ is any rotation number, for any $c\in\rho^{-1}(R)$ and any weak K.A.M. solution $u$ for $\widehat T^c$, the set of invariant Borel probability measures with support in $\cg(c+u')$ is independent from $c\in\rho^{-1}(R)$ and $u$.  Those measures are called Mather measures and the union of the supports of these measures is called the {\em Mather set} for $R$ and is denoted by $\cm(R)$;
\item We say that a function $u$ defined on a part $A$ of $\T$ is $c$-dominated if,    denoting by $\widetilde A$ the lift of $A$ to $\R$,  and $\tilde u$ a lift of $u$, we have
$$\forall \theta, \theta'\in \widetilde A,\forall n\geq 1, \quad  \tilde u(\theta)-\tilde u(\theta')\leq \cs^c_n(\theta', \theta)+ n\alpha(c);$$
a weak K.A.M. solution for $\widehat T^c$ is always $c$-dominated; {if $A=\T$ a function $u : \T\to \R$ is $c$-dominated  if and only if
$$\forall \theta, \theta'\in \R, \quad  \tilde u(\theta)-\tilde u(\theta')\leq S(\theta', \theta)+c(\theta'-\theta)+\alpha(c);$$}
\item\label{K.A.M.ext} If $u:\cm\big(\rho(c)\big)\rightarrow \R$ is dominated, then there exists only one extension $U$ of $u$ to $\T$ that is a weak K.A.M. solution for $\widehat T^c$.   This function is   given by
$$\forall x\in \T, \quad U(x) = \inf_{y\in  \cm\textrm{$\big($}\rho(c)\textrm{$\big)$}} u(y) + S^c(y,x),$$
where $S^c $ is given by   Equation \eqref{projS}. The proof of this last point uses very standard ideas. As we have not found it exactly stated in this way in the literature, we provide a sketch of proof for the reader's convenience in  appendix \ref{appendix-3}.
\end{enumerate}

\subsection{Proof of Theorem \ref{Tgenecont}} 
  When there is no ambiguity in the notations, we will put $\sim$ signs to signify that we consider lifts of functions defined on $\T$.
 We will need the following lemma.
\begin{lemma}\label{Llimweak} Let $(c_n)$ be a  sequence of real numbers convergent to $c$ and let $(u_{c_n})$ be a  sequence of functions uniformly convergent to $v$ such that $u_{c_n}$ is  a weak K.A.M. solution for $\widehat T^{c_n}$. Then $\displaystyle{\lim_{n\rightarrow \infty} u_{c_n}}$ is a weak K.A.M. solution for $\widehat T^c$.
\end{lemma}
\begin{proof}
We know from Equation (\ref{Esolweakk}) that 
$$\tilde u_{c_n}(\theta)=\inf_{\theta'\in\R} \big(\tilde u_{c_n}(\theta')+S(\theta', \theta)+c_n(\theta'-\theta)+\alpha(c_n)\big).$$
Because of the superlinearity of $S$ and the fact that the $u_{c_n}$ and $c_n$ are uniformly bounded, there exists a fixed compact set $I$ in $\R$ such that for every $n$, we have
$$\tilde u_{c_n}(\theta)=\inf_{\theta'\in I} \big(\tilde u_{c_n}(\theta')+S(\theta', \theta)+c_n(\theta'-\theta)+\alpha(c_n)\big).$$
We  deduce from the uniform convergence of $(u_{c_n})$ to $v$ that
$$\tilde v(\theta)=\inf_{\theta'\in I} \big(\tilde v(\theta')+S(\theta', \theta)+c(\theta'-\theta)+\alpha(c)\big).$$
As we could do the same proof for $I$ as large as wanted, we have in fact
\begin{equation}\label{ElimwehK.A.M.} \tilde v(\theta)=\inf_{\theta'\in \R} \big(\tilde v(\theta')+S(\theta', \theta)+c(\theta'-\theta)+\alpha(c)\big).\end{equation}

\end{proof}
Let us now prove Theorem  \ref{Tgenecont}.\\
We have seen in subsection \ref{ssweakk} that we have only one possible choice for $u_c$ when $\rho(c)$ is irrational. 
\begin{nota}
We use the notation  $\ci=\{ c\in \R; \ \ \rho(c)\in \R\backslash \Q\}$.
\end{nota}
Let us prove that any extension $c\mapsto u_c$ that maps $c$ on a weak K.A.M. solution for $\widehat T^c$ that vanishes at $0$ is continuous at every $c\in \ci$. Let us consider a monotone sequence $(c_n)$  that converges to $c\in\ci$  and such that the sequence $\rho(c_n)$ is strictly monotone. Then the sequence $(c_n+u'_{c_n})$ is also monotone   
by Corollary \ref{Cordre}  and bounded  because the $u_{c_n}$ are equi semi-concave and then equiLipschitz
so convergent to a function $d$. We define for every $t\in [0, 1]$, $v(t)=\int_0^td(s)ds-ct$. Then we have $d=c+v'$   Lebesgue almost everywhere. As the sequence  $(c_n+u'_{c_n})$ is  bounded,   the Lebesgue dominated convergence Theorem implies that $(u_{c_n})$  pointwise converges to $v$. Because of the Ascoli Theorem, this convergence is uniform.\\
Because of Lemma \ref{Llimweak}, $v$ is a weak K.A.M. solution for $\widehat T^c$. As $v$ vanishes at $0$ and $\rho(c)$ is irrational, then $v$ is the unique weak K.A.M. solution for $\widehat T^c$ that vanishes at $0$, i.e. $v=u_c$.\\
 If now $(c_n)$ is any monotone sequence    that converges to $c\in\ci$, we can choose a  monotone sequence $(c'_n)$  that converges to $c\in\ci$ and satisfies
\begin{itemize}
\item  the sequence $\rho(c'_n)$ is a strictly monotone sequence;
\item for every $n\in \N$, there exists $k_n\in\N$ such that $\rho(c_n)$ is strictly between $\rho(c'_{k_n})$ and $\rho(c'_{k_n+1})$.
\end{itemize}
Then $(c_n+u'_{c_n})$ and $(c'_n+u'_{c'_n})$  have the same limit by Corollary \ref{Cordre} and we conclude as before that $(u_{c_n})$  uniformly converges to $u_c$.

This gives the wanted continuity at every point of $\ci$.\\

 Building  a function $u$, the only problem of continuity we have now to consider is at the points of the set $\rho^{-1}(\Q)$.\\
 Observe that if we find a continuous extension to $\R^2$ such that every $u_c$ is a weak K.A.M. solution for $\widehat T^c$, replacing $u_c$ by $u_c-u_c(0)$, we obtain an extension as wanted.

Let us now assume that $\frac{p}{q}$ is rational and let us introduce the notations $\rho^{-1}(\frac{p}{q})=[a_1,a_2]$,  $\ci_+=\ci\cap [a_2, +\infty)$ and $\ci_-=\ci\cap (-\infty, a_1]$. Observe that $a_1$ (resp. $a_2$) is a limit point of the set $\ci_-$ (resp. $\ci_+$). Let $(c_n)$ be a decreasing sequence in $\ci_+$ that converges to $a_2$.  Then by Proposition \ref{Pordreirrat}, $(c_n+u'_{c_n})_{n\in\N}$ is a decreasing sequence and then $\big(v_n:\theta\in[0, 1]\mapsto c_n\theta+\tilde u_{c_n}(\theta)\big)_{n\in\N}$ is also a decreasing sequence, thus convergent and even uniformly convergent by the Ascoli Theorem. By Lemma \ref{Llimweak}, $\displaystyle{ u_{a_2}(\theta)=\lim_{n\rightarrow \infty} v_n(\theta)-c_n\theta}$ defines a weak K.A.M. solution for $T^{a_2}$ such that $u_{a_2}(0)=0$.
Observe that we have in fact
$$\lim_{\substack{c\in \ci_+\\ c\rightarrow a_2}} u_{c}=u_{a_2}$$
because each such decreasing sequence $(c_n)$ defines a uniformly convergent sequence $(u_{c_n})$ and so the limit doesn't depend on the considered decreasing sequence.\\
In a similar way, we define a weak K.A.M. solution for $T^{a_1}$ by taking increasing sequences $(c_n)$
$$\lim_{\substack{c\in \ci_-\\ c\rightarrow a_1}} u_{c}=u_{a_1}.$$
Because $u_{a_1}$ and $u_{a_2}$ are weak K.A.M. solutions, they are dominated and we have
$$\forall x, y\in \R, \forall n\geq 1, \quad u_{a_i}(x)-u_{a_i}(y)\leq  \cs^{a_i}_n(x, y)+n\alpha (a_i).$$
Let $c=\lambda a_1+(1-\lambda) a_2\in [a_1, a_2]$. We use the notation $v_c=\lambda u_{a_1}+(1-\lambda) u_{a_2}$. Observe that $\alpha (c)=\lambda\alpha(a_1)+(1-\lambda)\alpha(a_2)$ because $\alpha'=\frac{p}{q}$ is constant on $[a_1, a_2]$. Then we have
$$\forall x, y\in \widetilde\cm\Big(\frac{p}{q}\Big),\quad v_c(y)-v_c(x)\leq \cs^c_n(x,y)+n\alpha(c);$$
i.e. $v_c$ is $c$-dominated on $\cm\big(\frac{p}{q}\big)$. We deduce from Point (\ref{K.A.M.ext}) of subsection \ref{ssmoreK.A.M.} that there exists only one extension   $u_c$ of  $v_{c}$ restricted to $\widetilde\cm\big(\frac{p}{q}\big)$ that is a weak K.A.M. solution for $\widehat T^c$.\\

Let us  prove that $c\in [a_1, a_2]\mapsto u_c$ is continuous. By definition of $u_c$, the map $c\mapsto u_{c|\cm(\frac{p}{q})}$ is continuous. Let us now consider a sequence $(c_n)$ in $[a_1, a_2]$ that converges to some $c\in [a_1, a_2]$. By Ascoli Theorem the set $\{ u_{c_n}, n\in\N\}$ is relatively compact for the uniform convergence distance. Let $U$ be a limit point of the sequence $(u_{c_n})$. By Lemma \ref{Llimweak}, we know that $U$ is a weak K.A.M. solution for $\widehat T^c$. Moreover, we have $U_{|\widetilde\cm(\frac{p}{q})}=u_{c|\widetilde\cm(\frac{p}{q})}$.  Using Point (\ref{K.A.M.ext}) of subsection \ref{ssmoreK.A.M.}, we deduce that $U=u_c$ and the wanted continuity.

To conclude that our choice is continuous everywhere, we only have to prove that $a_1$ is a continuity point from the left and that $a_2$ is a continuity point from the right. If we know that there is only one weak K.A.M. solution for $\widehat T^{a_i}$ that vanishes at $0$, we will conclude by the same argument we used for any $c\in\ci$.\\
Let us assume that $v$ is another weak K.A.M. solution for $\widehat T^{a_1}$ that vanishes at $0$. Because of Proposition \ref{Pordreirrat}, we have
$$\forall c<a_1,\quad c+u_c'<a_1+v'.$$
Taking the limit in this inequality and using the definition of $u_{a_1}$, we deduce that $v'\geq u'_{a_1}$. As $0=\int_\T v'=\int_\T u'_{a_1}$, we deduce that $u'_{a_1}=v'$ Lebesgue almost everywhere and then $v=u_{a_1}$.

At the end of the previous proof, we have actually established a fact that will be useful later:
\begin{propos}\label{unicite}
Let $R=\frac{p}{q}$ be a rational number and set $[a_1,a_2] = \rho^{-1}(R)$. Then, up to constants, there exists a unique weak K.A.M. solution for $\widehat T^{a_1}$ (resp. $\widehat T^{a_2}$).
\end{propos}

\subsection{More on the constructed function}

In this paragraph, $u : \A \to \R$ is any function given by Theorem \ref{Tgenecont} meaning that 
\begin{itemize}
\item  $u$ is continuous;
\item $u(0,c)=0$;
\item each $u_c= u(\cdot ,c)$ is a weak K.A.M. solution for the cohomology class $c$. 
\end{itemize}
We aim to give the range of the 
 map 
$(\theta,c) \mapsto \big(\theta, c+\frac{\partial u}{\partial \theta}(\theta,c)\big)$.  The following proposition asserts that any twist map is weakly integrable in the sense that $\A$ is covered by Lipschitz circles arising from weak K.A.M. solutions.

Recall that  $\mathcal{PG}(c+u'_c) =\{ (0,c)+\partial u_c (t), \quad t\in \T\}$ is the full pseudograph of $c+u'_c$.

\begin{propos}\label{+}
The following holds:
\begin{equation}\label{a0}
 \bigcup_{c\in \R} \mathcal{PG} (c+u'_c)=\bigcup_{\substack{t\in \T \\ c\in \R}} (0,c)+\partial u_c (t) = \A.
\end{equation}
\end{propos}
Let us define two auxiliary functions  with values in $\R\cup\{ +\infty, -\infty\}$:
$$ \forall \theta \in \T,\quad \eta_+(\theta) = \sup
 \big\{p\in \R ; \quad \exists c\in\R; \quad (\theta,p)\in \overline{\cg (c+u_c')}\big\},
$$
and 
$$ \forall \theta \in \T,\quad \eta_-(\theta) = \inf
 \big\{p\in \R ; \quad  \exists c\in\R; \quad (\theta,p)\in \overline{\cg (c+u_c')}\big\}.
$$
Finally define $\A_0 = \big\{(\theta,c)\in \A , \quad \eta_-(\theta)<c<\eta_+(\theta)\big\}$.

The following lemma is proved in Appendix \ref{ssLfullman}.

\begin{lemma}\label{Lfullman}
For all $c\in \R$, $\mathcal{PG}(c+u'_c)$ is a Lipschitz one dimensional compact manifold, hence it is an essential circle. 
\end{lemma}

It follows that the set $\A_0$ is open and connected (we will see at the end that it is in fact $\A$).
Indeed, by Jordan's theorem and Proposition \ref{Pordreirrat}, for $c<c'$ such that $\rho(c)<\rho(c')$, the set $\big\{(t,p)\in \A , \quad c+\partial u_c(t)<p<c'+\partial u_{c'} (t)\big\}$ is open and connected. Now $\A_0$ is an increasing union of such sets. 
 
\begin{propos}
The following equality holds:
$$\A_0 = \bigcup_{c\in \R} \mathcal{PG} (c+u'_c).$$
\end{propos}

\begin{proof}
We now denote by $\mathcal B$ the right hand side of  the previous equation. Observe that $\cb\subset \A_0$.
First we prove that $\mathcal B$ is closed in $\A_0$. To this end, let $(t_n,p_n)\in \mathcal{PG}(c_n+u'_{c_n})$ be a sequence converging to $(t,p)\in \A_0$. By definition of $\A_0$, there are $C_0<C_1$ such that $P_0<p<P_1$ where $P_0$ and $P_1$ are such that $(t,P_0) \in \mathcal{PG}(C_0+u'_{C_0})$ and $(t,P_1) \in \mathcal{PG}(C_1+u'_{C_1})$. 
Now let $c_-<C_0<C_1<c_+$ be such that $\rho(c_-)<\rho(C_0)$, $\rho(c_+)>\rho(C_1)$  and $\rho(c_-)$, $\rho(c_+)$ are irrational. As the pseudographs are vertically ordered (Proposition \ref{Pordreirrat}), $(t,p)$ is trapped in the open sub-annulus between $\mathcal{PG}(c_-+u'_{c_-})$ and $\mathcal{PG}(c_++u'_{c_+})$. It follows that for $n$ large enough, so is $(t_n,p_n)$.   Hence $\mathcal{PG}(c_n+u'_{c_n})$ is a full pseudograph that contains a point strictly between  $\mathcal{PG}(c_-+u'_{c_-})$and $\mathcal{PG}(c_++u'_{c_+})$. Proposition \ref{Pordreirrat} implies that  $\rho(c_-)\leq\rho(c_n)\leq \rho (c_+)$. As $\rho(c_-)$, $\rho(c_+)$ are irrational, there is a unique weak K.A.M. solution for these rotation numbers and then $\rho(c_n)\notin \{ \rho(c_-), \rho (c_+)\}$. We deduce that
$\rho(c_-)<\rho(c_n)< \rho (c_+)$  and then that $c_-<c_n<c_+$.

Up to extracting, we may assume that $c_n \to c_\infty$ and by continuity of the pseudographs  with respect to $c$ (for the Hausdorff distance, see Proposition \ref{hausdorff}), it follows that $(t,p)\in \mathcal{PG}(c_\infty+u'_{c_\infty})\subset \A_0$.

Next we prove that $\mathcal B=\A_0$. We argue by contradiction, by the first part, if this is not the case, there is an open ball $B=(\theta_0,\theta_1)\times (r_0,r_1)$ such that $\overline B\subset \A_0 \setminus \mathcal B$. 

We say that a topological essential circle $\mathcal C$ is above $B$ if $B$ is included in the lower connected component of $\A \setminus \mathcal C$\footnote{Recall that by Jordan's theorem, $\A \setminus \mathcal C$ has two open connected components, one we call upper that contains $\T \times (k,+\infty)$  and one  we call lower, that contains $\T \times (-\infty, -k)$ for $k$ large enough.} and $\mathcal C$ is under $B$ if $B$ is included in the upper connected component of $\A \setminus \mathcal C$. Therefore, if we set $EC_B$ the set of essential circles  of the family $\mathcal C \subset \A\setminus B$, $EC_B$ is the union of circles above $B$:  $EC_B^+$ and those under $B$: $EC_B^-$.

We will prove that 
\begin{lemma}
Both $EC_B^+$ and $EC_B^-$ are open subsets of $EC_B$ for the Hausdorff distance.
\end{lemma}
\begin{proof}
We prove it for $EC_B^+$. Let $\mathcal C^+$ be a circle above $B$. As the lower connected component of $\A \setminus \mathcal C^+$ is path connected, there is a continuous  path $\gamma : [0,+\infty) \to \A \backslash \cC^+$ such that $\gamma(0)\in B$ and $\gamma(t) = (0,-t)$ for $t$ large enough. Let $\varepsilon>0$ be such that  $\mathcal C^+$ is at distance greater than $\varepsilon$ from $\gamma$. If $\mathcal C^-$ is any circle under $B$, then it must intersect $\gamma$. Hence $d_H(\mathcal C^-,\mathcal C^+)>\varepsilon$ where $d_H$ stands for the Hausdorff distance. This proves the lemma. 
\end{proof}
We will obtain a contradiction as $\R$ is connected and the map $c\mapsto \mathcal{PG}(c+u'_c)$ is continuous for the Hausdorff distance, provided  we prove that for $c$ large, $\mathcal{PG}(c+u'_c)$ is above $B$ while for $c$ small $\mathcal{PG}(c+u'_c)$ is under $B$.

\begin{lemma}
For $c$ large, $\mathcal{PG}(c+u'_c)$ is above $B$ while for $c$ small $\mathcal{PG}(c+u'_c)$ is under $B$.
\end{lemma}
\begin{proof}
We establish only the first fact. Let $\theta\in (\theta_0,\theta_1)$. By definition of $\eta_+$, there exists $C$ such that for $c>C$, then $p>r_1$ for all $p$ verifying $(\theta,p) \in \mathcal{PG} (c+u_c')$. For $t>0$ small it follows that  $p>r_1$ for all $p$ verifying $(\theta,p) \in \varphi_{-t}\big(\mathcal{PG} (c+u_c')\big)$, where $\varphi$ denotes here the flow of the pendulum. Moreover, up to taking $t$ smaller, we may require $\varphi_{-t}\big(\mathcal{PG} (c+u_c')\big)$ disjoint from $B$. But it is proved in \cite{Arna2} that $\varphi_{-t}\big(\mathcal{PG} (c+u_c')\big)$ is the Lipschitz graph for small $t>0$ of a function $\alpha_t : \T \to \R$. Hence it follows from the intermediate value theorem that $\alpha(\theta) >r_1$ for $\theta \in (\theta_0, \theta_1)$ and it becomes obvious that  $B = (\theta_0,\theta_1)\times (r_0,r_1)$ is under $\varphi_{-t}\big(\mathcal{PG} (c+u_c')\big)$ and letting $t\to 0$ and passing to the limit, we obtain that $B$ is under $\mathcal{PG} (c+u_c')$.
\end{proof}

\end{proof}

In order to conclude, we have to prove that $\A = \A_0$ which is equivalent to proving that $\eta_+$ is identically $+\infty $ and $\eta_-$ is identically $-\infty$. We will establish the result for $u_+$.

\begin{lemma}\label{Lhorscomp}
Let $[a,b]$ be a segment, there exists $C >0$ depending on 
 $[a, b]$ such that if $|c|>C$ then 
$$ \forall \theta\in [0,1], \theta' \in [a,b], \quad  S(\theta,\theta') + c(\theta-\theta') > \min_{n\in \Z}S(\theta,\theta'+n) + c(\theta-\theta'-n).$$

\end{lemma}
\begin{proof}
Let us set $\Delta = \max \left\{ \Big|\frac{\partial S}{\partial \theta'}(\theta,\theta')\Big| , \ \ \theta\in [0,1], \theta' \in [a-1, b+1]\right\}$ and $C=\Delta +1$.

If $|c|>C$ two cases may occur:
\begin{itemize}
\item either $c>\Delta+1$. In this case, if $(\theta,\theta')\in  [0,1]\times [a,b]$, by Taylor's inequality we find 
$$S(\theta,\theta')+ c(\theta-\theta') >S(\theta,\theta')+ c\big(\theta-(\theta'+1)\big)+\Delta \geq  S(\theta,\theta'+1)+ c\big(\theta-(\theta'+1)\big);$$
\item or $c <-\Delta -1$, in which case 
$$S(\theta,\theta')+ c(\theta-\theta') >S(\theta,\theta')+ c\big(\theta-(\theta'-1)\big)+\Delta \geq  S(\theta,\theta'-1)+ c\big(\theta-(\theta'-1)\big).$$

\end{itemize}

\end{proof}

\begin{cor}
The function $\eta_+$ is identically $+\infty$.
\end{cor}
\begin{proof}
Let  us fix $A>0$. We assume that for all $(\theta,p)\in  \mathcal{PG}(u'_0)$, then $|p|\leq A$ (or in other words, $u_0$ is $A$-Lipschitz).  As every map $\theta\mapsto \frac{\partial S}{\partial \Theta}(\theta, \Theta_0)$ is a decreasing diffeomorphism of $\R$, there exists a constant $B>0$ such that for every $\Theta_0\in [0, 1]$, we have
$$\theta>B\Rightarrow  \frac{\partial S}{\partial \Theta}(\theta, \Theta_0)<-(A+1)\quad{\rm and}\quad \theta<-B\Rightarrow \frac{\partial S}{\partial \Theta}(\theta, \Theta_0)>A+1.$$
Let $C$ be the constant given by Lemma \ref{Lhorscomp} for the segment $[-B,B]$  and let us choose   $c>\sup\{ B, C\}$.    Let $\theta_0\in [0, 1]$ be any derivability point of $u_c$.   Because of Lemma \ref{Lhorscomp}, if $\tilde u_c$ is a lift of $u_c$ and if $\tilde\theta$ verifies 
$$\tilde u_c (\theta_0) = \inf_{\theta\in \R} \tilde u_c(\theta) + S(\theta,\theta_0) +c(\theta-\theta_0)= \tilde  u_c(\tilde\theta) + S(\tilde\theta,\theta_0) + c(\tilde\theta-\theta_0),$$
then $\tilde \theta \notin [-B,B]$ and then $\Big|\frac{\partial S}{\partial \Theta}(\tilde\theta, \theta_0)\Big|>A+1$.

 We deduce from point \ref{ptdifftt} of section \ref{ssweakk} that $f\big(\tilde\theta, c+u_c'(\tilde \theta)\big)=\big(\theta_0, c+u_c'(\theta_0)\big)$ and then
 $$c+\tilde u_c'(\theta_0) =   \frac{\partial S}{\partial \Theta} (\tilde\theta,\theta_0),$$
and then $|c+\tilde u_c'(\theta_0)| >A+1$.\\
As  $\int_0^1 \big(c+\tilde u'_c(s)\big) ds = c >0$, we can choose $\theta_0$ such that  $c+\tilde u'_c(\theta_0) >0$ and so $c+\tilde u_c'(\theta_0)>A+1$. 

As the pseudographs are vertically ordered (Proposition \ref{Pordreirrat}), $\mathcal{PG}(c+u'_c)$ is above $\mathcal{PG}(u'_0)$. We conclude that for all derivability point $\theta$ of $u_c $ then $c+\tilde u_c'(\theta) >A+1$. Finally, the whole full pseudograph $\mathcal{PG}(c+u'_c)$ lies above the circle $\{(t,A) , \ \ t\in \T\}$.



We have just established that if $c> B$, then $\mathcal{PG}(c+u'_c)$ lies above the circle $\{(t,A) , \ \ t\in \T\}$, that concludes the proof.

\end{proof}

\section{Proof of the implication (\ref{pt1Tgeneder}) $\Rightarrow$  (\ref{pt2Tgeneder})  in   Theorem \ref{Tgeneder}}\label{secfirst}

We assume that $f:\A\rightarrow \A$ is a $C^k$ symplectic twist map (with $k\geq 1$) that has a continuous invariant foliation into continuous graphs $a\in\R\mapsto  \eta_a\in C^{0}(\T, \R)$ where we choose $\eta_a(0)=a$. Then Birkhoff's theorem (see \cite {Bir1}, \cite{Fa1}  and \cite {He1}) implies that all the $\eta_a$ are Lipschitz.\\
\begin{nota}
For every $a\in\R$, we will denote by $g_a:\T\rightarrow \T$ the restricted-projected Dynamics to the graph of $\eta_a$, i.e
$$g_a(\theta)=\pi_1\circ f\big(\theta, \eta_a(\theta)\big).$$
\end{nota}
\subsection{Some generalities}\label{ss31}\hglue 1 truemm

\begin{nota}
\begin{itemize}
\item In $\R^2$ we denote by $B(x, r)$ the open disc for the usual Euclidean distance with center $x$ and radius $r$;
\item we denote by $R_\alpha:\T\rightarrow \T$ the rotation $R_\alpha(\theta)=\theta+\alpha$;
\item if E is a finite set, $\sharp(E)$ is the number of elements it contains;
\item we denote by $\lfloor \cdot\rfloor : \R\rightarrow \Z$ the integer part.
\end{itemize}
\end{nota}
\begin{defi}
\begin{itemize}
\item We say that $a\mapsto\eta_a$ defines a   {\em Lipschitz foliation} if $(\theta, a)\mapsto \big(\theta, \eta_a(\theta)\big)$ is an homeomorphism that is locally biLipschitz; if $f$ has an invariant Lipschitz foliation, $f$  is {\em Lipschitz integrable};
\item we say that $a\mapsto\eta_a$ defines a   {\em $C^k$ foliation} if $(\theta, a)\mapsto \big(\theta, \eta_a(\theta)\big)$ is a $C^k$ diffeomorphism;  if $f$ has an invariant $C^k$ foliation, $f$  is {\em $C^k$ integrable};
\item we say that $a\mapsto\eta_a$ defines a   {\em $C^k$ lamination} if $(\theta, a)\mapsto \big(\theta, \eta_a(\theta)\big)$ is an homeomorphism,  every $\eta_a$ is $C^k$ and the map $a\mapsto \eta_a$ is continuous when $C^k(\T, \R)$ is endowed with the $C^k$ topology.
\end{itemize}
\end{defi} 
\begin{proposition}
Assume that the $C^1$ symplectic twist map $f:\A\rightarrow \A$ has an invariant continuous (resp. locally Lipschitz continuous) foliation into graphs $a\in\R\mapsto  \eta_a\in C^{0}(\T, \R)$. Then the map $\ca: a\in\R\mapsto \int_\T \eta_a(t)dt\in\R$ is an homeomorphism (resp. locally biLipschitz homeomorphism).
\end{proposition}
The proof is straightforward. Using this result, we can use $c=\ca(a)$ instead of $a$ as a parameter, what we do from now.\\
\begin{notas}
We fix a lift $F:\R^2\rightarrow \R^2$ of $f$.   We denote by $\tilde \eta_c:\R\rightarrow \R$ the lift of $\eta_c$.
We denote by $\rho$ the function that maps $c\in \R$ to the rotation number $\rho(c)\in\R$ of the restriction of $F$ to the graph of   $\tilde\eta_c$.
\end{notas}
 The map $\rho$  is then an increasing homeomorphism. \\
When moreover the foliation is biLipschitz, we will prove that $\rho$ is a biLipschitz homeomorphism (see Proposition \ref{rho}).

  We recall a well-known result concerning the link between invariant measures and semi-conjugations for orientation preserving homeomorphisms of $\T$.

\begin{proposition}\label{Pconjmeas} Assume that $\mu_c$ is a non-atomic Borel invariant probability measure by $g_c$. Then,  if $\rho(c)$  is irrational or $g_c$ is $C^0$ conjugate to a rotation,
 the map $h_c:\T\rightarrow \T$ defined by $h_c(\theta)=\int_0^\theta d\mu_c$ is a semi-conjugation between $g_c$ and the rotation with angle $\rho(c)$, i.e:
$$h_c\big(g_c(\theta)\big)=h_c(\theta)+\rho(c).$$

\end{proposition}

\begin{proof}
Let $\tilde \mu_c$ be the pull back  measure of $\mu_c$  to $\R$ and let $\tilde g_c:\R\rightarrow \R$ be a lift of $g_c$ to $\R$. Then we have for every $\Theta\in [0, 1]$ lift of $\theta\in \T$:
$$\tilde \mu_c([0, \Theta])=\tilde \mu_c([\tilde g_c(0), \tilde g_c(\Theta)])=\tilde \mu_c\big(\big[ \lfloor \tilde g_c(0)\rfloor, \tilde g_c(\Theta)\big]\big)-\tilde \mu_c\big(\big[\lfloor \tilde g_c(0)\rfloor, \tilde g_c(0)\big]\big);$$
where $\lfloor\tilde g_c(0)\rfloor$ is the integer part of $\tilde g_c(0)$. This implies\footnote{ { Recall that if $f : \T \to \T$ is an orientation preserving homeomorphism then either $\rho(f)$ is irrational, $f$ is semi-conjugated (by $h$) to the rotation $R_{\rho(f)}$  and the only invariant measure is the pull back of the Lebesgue measure by $h$; or $\rho(f)$ is rational and the invariant measures are supported on periodic orbits.  When $\rho(f)$ is irrational or when $f$ is $C^0$ conjugate to a rotation, then for any invariant measure $\mu$ and $x\in \T$, $\mu([x,f(x)[) = \rho(f)$. }  }
$$h_c(\theta)=h_c\big(g_c(\theta)\big)-\tilde\mu_c\big(\big[0, g_c(0)\big]\big)=h_c\big(g_c(\theta)\big)-\rho(c).$$
Moreover, as we assumed that $\mu_c$ is non-atomic, $h_c$ is continuous.
\end{proof}

\begin{remks}
\begin{enumerate}
\item In the other sense, if $h_c$ is a  (non-decreasing) semi-conjugation such that $h_c\circ g_c=h_c+\rho(c)$, then $\mu([0, \theta])=h_c(\theta)-h_c(0)$ defines a $g_c$-invariant Borel probability measure;
\item When $\rho(c)$ is irrational, it is well known that the Borel invariant probability measure $\mu_c$ is unique and that the semi-conjugation $h_c$ is unique up to constant.  
\end{enumerate}
\end{remks}
\begin{nota}
When $\rho(c)$ is irrational, we will denote by $h_c$ the semi-conjugation such that $h_c(0)=0$.
\end{nota}


Before entering the core of the proof, let us mention a useful fact about iterates of $C^0$-integrable twist maps:
\begin{proposition}\label{iterateTwist}
Let $f: \A \to \A$ be a $C^0$-integrable twist map, then so is $f^n$ for all $n>0$.
\end{proposition}
 This is specific to the integrable case: { in general, an iterated twist map is not a twist map as can be seen in the neighborhood of an elliptic fixed point.}
\begin{proof}
We argue by induction on $n>0$. The initialization being trivial, let us assume the result true for some $k>0$. Let $F : \R^2 \to \R^2$ be a lift of $f$. For any $c\in \R$ using the notations given at the beginning of section \ref{secfirst}, we have 
$$\forall \theta \in \T, \forall m>0 , \quad f^m\big(\theta, \eta_c(\theta)\big) = \big(g_c^m(\theta), \eta_c\circ g_c^m(\theta)\big).$$
 Observe that if $f^m$ satisfies the twist condition and $c_1<c_2$ are two real numbers, then we have
$$\tilde g^m_{c_1}(t)=\pi_1\circ F^m\big(t, \eta_{c_1}(t)\big)<\pi_1\circ F^m\big(t, \eta_{c_2}(t)\big)=\tilde g^m_{c_2}(t)$$
and $\displaystyle{\lim_{t\rightarrow\pm\infty}\tilde g_{c_1}(t)=\pm\infty}$.

 Let us prove this. Let $c_1<c_2$ and $t\in \R$. Denoting with $\sim$ the lifts of the considered functions we obtain that 
$$ \pi_1\big(F^{n+1}(t,c_2)\big) -\pi_1\big(F^{n+1}(t,c_1)\big) = \tilde g_{c_2}\circ \tilde g^{n}_{c_2}(t) -\tilde g_{c_1}\circ  \tilde g^{n}_{c_1}(t) \geq  \tilde g_{c_2}\circ \tilde g^{n}_{c_1}(t) -\tilde g_{c_1}\circ \tilde g^{n}_{c_1}(t) ,$$
where we have used the induction hypothesis and the fact that $\tilde g_{c_2}$ is increasing. It follows that $c\mapsto  \pi_1\big(F^{n+1}(t,c)\big) $ is an increasing diffeomorphism on its image.  Observe also that this inequality implies that $\displaystyle{\lim_{c_2\rightarrow+\infty} \pi_1\big(F^{n+1}(t,c_2)\big) =+\infty}$ because $\displaystyle{\lim_{c_2\rightarrow +\infty}\tilde g_{c_2}(s) =+\infty}$. Moreover 
$$ \pi_1\big(F^{n+1}(t,c_2)\big) -\pi_1\big(F^{n+1}(t,c_1)\big) = \tilde g_{c_2}\circ \tilde g^{n}_{c_2}(t) -\tilde g_{c_1}\circ  \tilde g^{n}_{c_1}(t) \leq  \tilde g_{c_2}\circ \tilde g^{n}_{c_2}(t) -\tilde g_{c_1}\circ \tilde g^{n}_{c_2}(t) ,$$
implies that $\displaystyle{\lim_{c_1\rightarrow-\infty} \pi_1\big(F^{n+1}(t,c_1)\big) =-\infty}$ because $\displaystyle{\lim_{c_1\rightarrow -\infty}\tilde g_{c_1}(s) =-\infty}$. So finally $c\mapsto  \pi_1\big(F^{n+1}(t,c)\big) $ is an increasing diffeomorphism onto $\R$.

\end{proof}

\subsection{Differentiability and conjugation along the rational curves}\label{ssrat}
It is proved in \cite{Arna1} that for every $r=\frac{p}{q}\in\Q$, $\eta_c=\eta_{\rho^{-1}(r)}$ is   $C^k$   and the restriction of $f$ to the graph $\Gamma_c$ of $\eta_c$ is completely periodic: $f^q_{|\Gamma_c}={\rm Id}_{\Gamma_c}$. Moreover, along these particular curves, the two Green bundles (see Appendix  \ref{ssGreenb} for definition and results) are equal:
$$G_-\big(\theta, \eta_c(\theta)\big)=G_+\big(\theta, \eta_c(\theta)\big).$$
  \begin{theorem}\label{Tdiffrat}
  \begin{itemize}
\item Along every leaf $\Gamma_c$ such that $\rho(c)\in\Q$, the derivative $\frac{\partial \eta_c(\theta)}{\partial c}=1+\frac{\partial^2u_c}{\partial c\partial \theta}>0$ exists and  $C^{k-1}$  depends on $\theta$;
\item for any $c$ such that $\rho(c)$ is rational, the measure $\mu_c$ on $\T$ with density $\frac{\partial \eta_c}{\partial c}$ is a Borel probability measure invariant by $g_c$ and  for $\theta\in [0, 1]$, the equality
$$h_c(\theta)=\mu_c([0, \theta])=\theta+\frac{\partial u}{\partial c}(\theta, c)$$ 
defines a conjugation between $g_c$ and the rotation with angle $\rho(c)$;
\item then the map $c\in\R\mapsto \mu_c$ is continuous and also $c\in\R\mapsto h_c$ for the uniform $C^0$ topology. Thus $(\theta, c)\mapsto h_c(\theta)$ is continuous.
\end{itemize}
\end{theorem}

\begin{remks}
\begin{enumerate}
\item Observe that because $c\mapsto \eta_c$ is increasing, we know that for Lebesgue almost every $(\theta, c)\in \T\times \R$, the derivative $\frac{\partial \eta_c(\theta)}{\partial c}$ exists (see \cite{KFT}). But our theorem says something different.
\item Because of the continuous dependence on $\theta$  along the rational curve,  we obtain that $\frac{\partial \eta_c(\theta)}{\partial c}$ restricted to every rational curve is bounded (that is clear when we assume that the foliation is Lipschitz but not if the foliation is just continuous).
\end{enumerate}
\end{remks}

\noindent{\em Proof of the first point.}
  We fix $A\in \R$ such that $\rho(A)=\frac{p}{q}\in \Q$. Replacing $f$ by $f^q$, we can assume that $\rho(A)\in\Z$. Observe that because of the $C^0$-integrability of $f$, $f^q$ is also a ($C^0$-integrable with the same invariant foliation) twist map (Proposition \ref{iterateTwist}).

We define $G_A: \T\times \R\rightarrow \T\times \R$ by 
\begin{equation}\label{Ecurvezero} G_A(\theta, r)=\big(\theta, r+\eta_A(\theta)\big).\end{equation}
Then $G_A^{-1}\circ f^q\circ G_A$ is also a $C^0$-integrable $C^k$ twist map and $\T\times\{ 0\}$ is filled with fixed points.

We finally have to prove our theorem in this case and we use the notation $f$ instead of    $G_A^{-1}\circ f^q\circ G_A$. We can assume that $A=0$  instead of $A\in\Z$.

Because of the semi-continuity of the two Green bundles $G_-=\R(1, s_-)$ and $G_+=\R(1, s_+)$, we have for any point $x=(\theta, r)$ sufficiently close to $\T\times \{ 0\}$: $\max\{ |s_-(x)|, |s_+(x)|\} <\varepsilon$ is small.

Now we fix $c$ small and consider for every $\theta\in\T$ the small triangular domain $\ct(\theta)$ that is delimited by the three following red curves

\begin{itemize}
\item the graph of $\eta_c$;
\item the vertical $\cv_\theta=\{\theta\}\times \R$;
\item the image $f(\cv_\theta)$ of the vertical at $\theta$.
\end{itemize}
\begin{center}
\includegraphics[width=5cm]{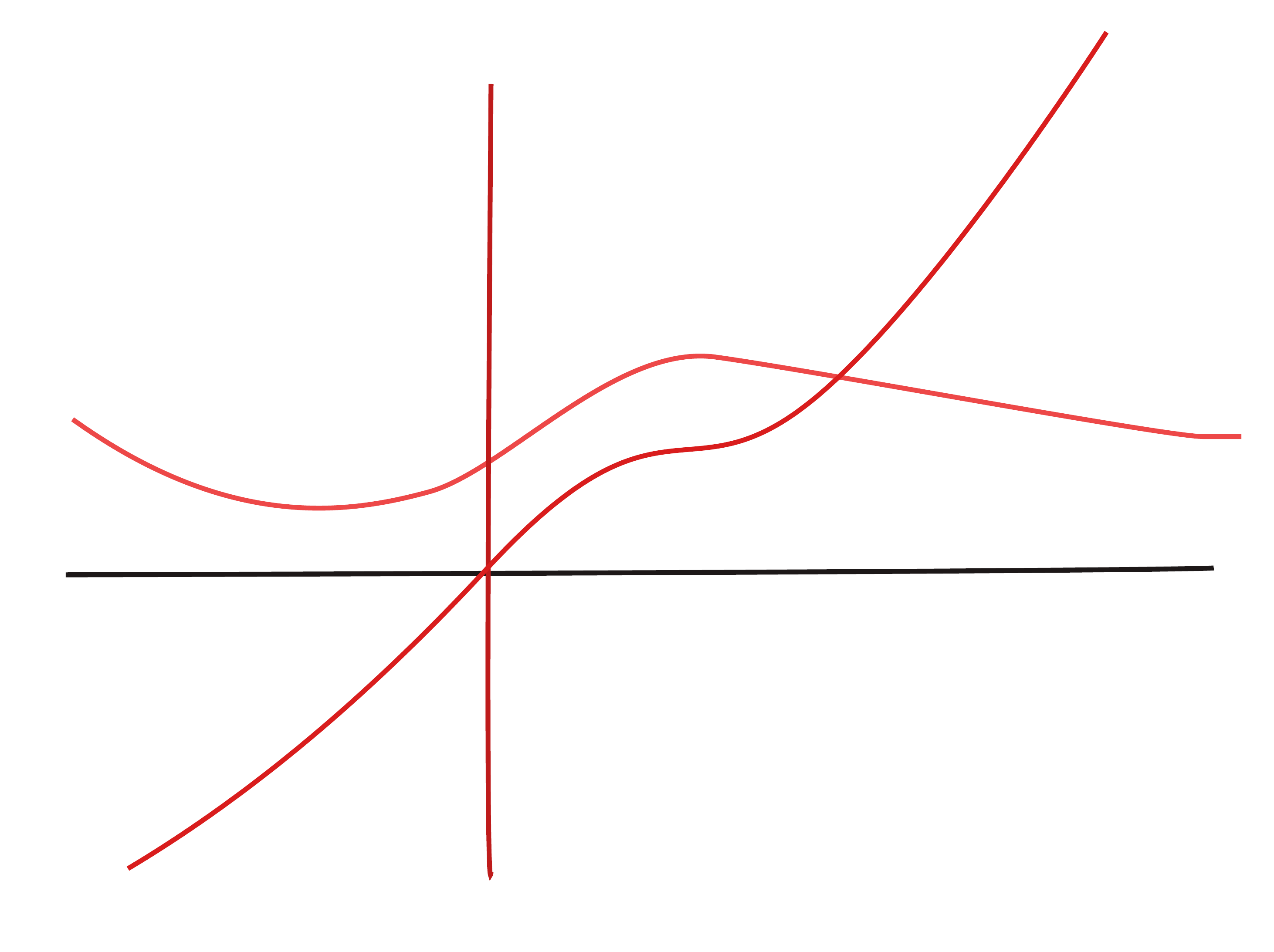}
\end{center}
 To be more precise, $\ct(\theta)$ is `semi-open' in the following sense; it contains  its whole boundary except the image $f(\cv_\theta)$ of the vertical at $\theta$.\\
We assume that $c>0$. The case $c<0$ is similar.\\
As the slope of $\eta_c$ is almost $0$ (because between the slope of the two Green bundles,  see Proposition \ref{PGreensand}) and the slope of the side of the triangle that is in $f(\cv_\theta)$ is almost $\frac{1}{s(\theta)}$ where $s(\theta)>0$ is the torsion that is defined by
\begin{equation}\label{Ematf} Df(\theta, 0)=\begin{pmatrix} 1&s(\theta)\\
0& 1\end{pmatrix},\end{equation}
the area of this triangle is 
\begin{equation}\label{E1}\lambda\big(\ct(\theta)\big)=\frac{1}{2}\big( \eta_c(\theta)\big)^2\big( s(\theta)+\varepsilon (\theta, c)\big);\end{equation}
where 
\begin{equation}\label{E2}{\rm uniformly}\quad{\rm for}\quad \theta \in \T,\quad  \lim_{c\rightarrow 0}\varepsilon (\theta, c)=0.\end{equation}
Let $\lambda$ be the Lebesgue measure restricted to the  invariant sub-annulus 
$$\ca_c=\bigcup_{\theta\in\T} \{ \theta\}\times [0, \eta_c(\theta)].$$
Being symplectic,   $f$ preserves $\lambda$. Moreover, every ergodic measure $\mu$ for $f$ with support in $\ca_c$ is supported on one curve $\Gamma_A$ with $A\in [0, c]$. But $f_{|\Gamma_A}$ is semi-conjugated to a rotation with an  angle  $\rho(A)$ that is close to $0$. Hence every interval in $\Gamma_A$ that is between some $\big(\theta, \eta_A(\theta)\big)$ and $f\big(\theta, \eta_A(\theta)\big)$ has the same $\mu$-measure, which is just given by the rotation number $\rho(A)$  on the graph of $\eta_A$. This implies that $\theta\mapsto \mu\big(\ct(\theta)\big)$ is constant. Hence for every $\theta, \theta'\in \T$ and for every ergodic measure with  support in $\ca_c$, we have $\mu\big(\ct(\theta)\big)=\mu\big(\ct(\theta')\big)$. Using the ergodic decomposition of invariant measures (see e.g. \cite{Man1}) $\lambda=\int\mu_ad\nu(a)$, we deduce that:
\begin{equation}\label{Airetriangle}\forall \theta, \theta'\in \T,\quad  \lambda\big(\ct(\theta)\big)=\lambda\big(\ct(\theta')\big)=\int \rho(a)d\nu(a).\end{equation}

We deduce from equations (\ref{E1}) and (\ref{E2}) that 
$${\rm uniformly}\ \ {\rm for}\quad \theta, \theta'\in\T,\quad  \lim_{c\rightarrow 0} \frac{\eta_c(\theta')}{\eta_c(\theta)}=\sqrt{\frac{s(\theta)}{s(\theta')}}.$$
 Integrating with respect to $\theta'$, we deduce that uniformly in $\theta$, we have
$$ \lim_{c\rightarrow 0} \frac{c}{\eta_c(\theta)}=\sqrt{s(\theta)}\int_\T \frac{dt}{\sqrt{s(t)}}.$$

This implies that 
\begin{equation}\label{neq 0}
\frac{\partial \eta_c(\theta)}{\partial c}_{|c=0} =\Big( \int_\T \frac{dt}{\sqrt{s(t)}} \Big)^{-1}\frac{1}{\sqrt{s(\theta)}};
\end{equation}
and even
\begin{equation}\label{Eeta} \eta_c(\theta)=c\Big( \int_\T \frac{dt}{\sqrt{s(t)}} \Big)^{-1}\bigg(\frac{1}{\sqrt{s(\theta)}}+\varepsilon(\theta,c)\bigg)
\end{equation}
where
\begin{equation}\label{Eetaunif}{\rm uniformly}\ \ {\rm for}\quad \theta \in \T,\quad \lim_{c\rightarrow 0}\varepsilon (\theta, c)=0.\end{equation}

Observe that $ \frac{\partial \eta_c}{\partial c}=\Big( \int_\T \frac{dt}{\sqrt{s(t)}} \Big)^{-1}\frac{1}{\sqrt{s( \cdot)}}$ is a  $C^{k-1}$  function of $\theta'$. This proves the first point of theorem \ref{Tdiffrat}.

\noindent{\em Proof of the second point.} We deduce  from the first point that for any $c$ such that $\rho(c)$ is rational, the function  $\frac{\partial \eta_c}{\partial c} $ is continuous and positive. Moreover, its integral on $\T$ is 1. Hence $\frac{\partial \eta_c}{\partial c} $  is the density of a Borel probability measure that is equivalent to Lebesgue.  We now introduce:

\begin{nota}
If $c<c'$, we denote by $\Lambda_{c, c'}$ the normalized Lebesgue measure between the graph of $\eta_c$ and the graph of $\eta_{c'}$.
\end{nota}
Then $f$ preserves $\Lambda_{c, c'}$. Observe that   for any measurable $I\subset \T$, we have
\begin{equation}\label{E4}\Lambda_{c, c'} \big(\{ (\theta, r); \theta \in I, \  r\in [\eta_c(\theta), \eta_{c'}(\theta)]\}\big)=\frac{1}{c-c'}\int_I \big(\eta_c(\theta)-\eta_{c'}(\theta)\big)d\theta.
\end{equation}
\begin{lemma}\label{Lratmeas}
If $\rho(c)$ is rational, then $\displaystyle{\lim_{c'\rightarrow c}\Lambda_{c, c'}}$ is a measure supported on the graph of $\eta_c$ whose projected measure  $\mu_c$   has  density   $\frac{\partial \eta_c}{\partial c}$ with respect to Lebesgue  of $\T$.

Hence if $h_c(\theta)=\int_0^\theta  \frac{\partial \eta_c}{\partial c}(t)dt$, we have $$h_c\circ \pi_1\circ f\big(\theta, \eta_c(\theta)\big)=h_c(\theta)+\rho(c).$$
\end{lemma}

\begin{proof}
 Using  Equations (\ref{Eeta}) and (\ref{Eetaunif}),  we can take the limit in Equation (\ref{E4}) or more precisely for any $\psi\in C^0(\A, \R)$ in 
$$\int \psi(\theta, r)d\Lambda_{c, c'}(\theta, r)=\int_\T\frac{1}{c-c'}\bigg( \int_{\eta_{c'}(\theta)}^{\eta_c(\theta)}\psi(\theta, r)dr\bigg) d\theta$$
 and obtain that the limit is an invariant measure supported in the graph of $\eta_c$ whose projected measure $\mu_c$ has a density with respect to Lebesgue that is equal to
$\frac{\partial \eta_c}{\partial c}$. We then use Proposition \ref{Pconjmeas} to conclude that $h_c$ is the wanted conjugation.
 \end{proof}
\noindent   {\em Proof of the third point.} 
We noticed that when $\rho(c)$ is irrational, there is only one invariant Borel probability measure that is supported on the graph of $\eta_c$. This implies the continuity of the map $c\mapsto \mu_c$ at such a $c$. Let us look at what happens when $\rho(c)$ is rational.
 
\begin{proposition}
For every $c_0\in \R$ such that $\rho(c_0)$ is rational, for every $\theta\in [0, 1]$, we have
$$\lim_{c\rightarrow c_0} \mu_c([0, \theta])=\mu_{c_0}([0, \theta])$$
and the limit is uniform in $\theta$.
\end{proposition}
This joint with the continuity of $h_{c_0}$ implies the continuity of $(\theta, c)\mapsto h_c(\theta)$ at $(\theta, c_0)$.\\
\begin{proof}  In this proof, we will use different functions $\varepsilon_i(\tau, c)$ and all these functions will satisfy uniformly in $\tau$
$$\lim_{c\rightarrow 0}\varepsilon_i(\tau, c)=0.$$
  As in the proof of the first point of Theorem \ref{Tdiffrat},  we can assume that $u_{c_0}=0$ (and then $c_0=0$)  and $\rho(0)=0$.\\
  We fix $\varepsilon>0$. Because of the continuity of  $\rho$, we can choose $\alpha$ such that if $|c|<\alpha$, then $|\rho(c)|<\varepsilon$.\\
 Let us introduce the notation $N_c=\lfloor \frac{1}{\rho(c)}\rfloor$ for $c\not=0$.    Let us assume that $c>0$ and $\theta\in (0, 1]$. 
We also denote by $  \tilde g_c$ the lift of $g_c$ such that $ \tilde g_c(0)\in[0, 1)$ and by $M_c(\theta)$
$$M_c(\theta)=\sharp \{ j\in \N;\quad  \tilde g _{c}^j(0)\in [0, \theta]\}.$$
 Hence, $M_c(\theta)$ is the number of points of the orbit of $0$ under $\tilde g_c$ that belong to $[0, \theta]$.  Observe that $M_c(\theta)$ is non-decreasing with respect to $\theta$.\\
As $\eta_c>0$, any primitive $\cN_c$ of $\eta_c$ is increasing, hence $M_c(\theta)$ is also the number of $\tilde g^k(0)$ such that $\cN_c\big(\tilde g^k(0)\big)$ belongs to $[\cN_c(0), \cN_c(\theta)]$, i.e.
\begin{equation}
\label{EestMC}
\begin{matrix} M_c(\theta)&=\sharp \Big\{ j\in  \N;\quad  \int_0^{\tilde g_{c}^j(0)}\eta_{c}(t)dt\leq  \int_0^{\theta}\eta_{c}(t)dt\Big\}\\
& =\sup\Big\{ j\in  \N;\quad  \int_0^{\tilde g_{c}^j(0)}\eta_{c}(t)dt\leq  \int_0^{\theta}\eta_{c}(t)dt\Big\}.\end{matrix}
\end{equation}
 Note that $M_c(1)=N_c$ because $g_c$ has rotation number $\rho(c)$ and that   we have $\forall \theta\in (0, 1]$, $M_c(\theta)\leq N_c$ as $M_c$ is non decreasing. We have also\\
$$\mu_c([0, \theta])=\sum_{j=0}^{M_c(\theta)-1}\mu_c([\tilde g_c^j(0), \tilde g_c^{j+1}(0)[)+\mu_c([\tilde g^{M_c(\theta)}(0), \theta])$$
and thus
$\mu_c([0, \theta])=M_c(\theta)\rho(c) +\Delta\rho(c)$ with $\Delta\in [0, 1]$  because $[\tilde g^{M_c(\theta)}(0), \theta]\subset [\tilde g^{M_c(\theta)}(0), \tilde g^{M_c(\theta)+1}(0)[$.\\
Hence
 \begin{equation}\label{Emesure}\mu_c([0, \theta])\in [ M_c(\theta)\rho(c), M_c(\theta)\rho(c)+\rho(c)]\subset \left[ \frac{M_c(\theta)}{N_c+1}, \frac{M_c(\theta)+1}{N_c}\right] .\end{equation}
Hence to estimate the measure $\mu_c([0, \theta])$ we need a good estimate of the number of $j$  such that  $\tilde g_{c}^j(0)$ belongs to $[0, \theta]$.
We have proved in Equations (\ref{Eeta}) and (\ref{Eetaunif}) that 
\begin{equation}
\label{Eetabis} \eta_{c}(\tau)=\Big(\int_\T \frac{dt}{\sqrt{s(t)}}\Big)^{-1}\frac{c\big(1+\varepsilon_0(\tau,c)\big)}{\sqrt{s(\tau)}}.
\end{equation}

 We deduce from Equation (\ref{Ematf}) that  $\tilde g_{c}(\tau)= \tau+\big(s(\tau)+\varepsilon_1(\tau, c)\big)\eta_c(\tau)$ where uniformly in $\tau$, we have: $\displaystyle{\lim_{c\rightarrow 0}\varepsilon_1(\tau, c)=0}$ and then by Equation (\ref{Eetabis}):
\begin{equation}\label{Epetitair}\int_\tau^{\tilde g_{c}(\tau)}\eta_{c}(t)dt=\eta_c(\tau)^2\big(s(\tau)+\varepsilon_2(\tau, c)\big) =\frac{c^2\big(1+\varepsilon_3(\tau, c)\big)}{\Big(\int_\T \frac{dt}{\sqrt{s(t)}}\Big)^2} .
\end{equation}
This says that the area between $\tau$ and $\tilde g_c(\tau)$ that is limited by the zero section and the graph of $\eta_{c}$ is almost constant (i.e. doesn't depend a lot on $\tau$).\\

 We deduce from Equation (\ref{EestMC})  that 
$$\int_0^{\tilde g_c^{M_c(\theta)}(0)}\eta_c(t)dt\leq\int_0^\theta\eta_c(t)dt<\int_0^{\tilde g_c^{M_c(\theta)+1}(0)}\eta_c(t)dt.$$
Hence 
$$\sum_{j=0}^{M_c(\theta)-1}\int_{\tilde g^j(0)}^{\tilde g^{j+1}(0)}\eta_c(t)dt\leq \int_0^\theta\eta_c(t)dt\leq 
\sum_{j=0}^{M_c(\theta)}\int_{\tilde g^j(0)}^{\tilde g^{j+1}(0)}\eta_c(t)dt.$$
Using Equation (\ref{Epetitair}), we deduce that
$$M_c(\theta)\frac{ c^2\big(1+\varepsilon_4(\theta, c)\big)}{\Big(\int_\T\frac{dt}{\sqrt{s(t)}}\Big)^2}\leq\frac{c\big(1+\varepsilon_5(\theta, c)\big)}{\int_\T\frac{dt}{\sqrt{s(t)}}}\int_0^\theta \frac{dt}{\sqrt{ s(t)}} <(M_c(\theta)+1)\frac{ c^2\big(1+\varepsilon_6(\theta, c)\big)}{\Big(\int_\T\frac{dt}{\sqrt{s(t)}}\Big)^2},
$$
and then 
\begin{equation}\label{EMC}M_c (\theta)
=\left\lfloor \frac{1}{c}\bigg({\int_\T \frac{dt}{\sqrt{s(t)}}}\Big( \int_0^{\theta}\frac{dt}{\sqrt{s(t)}}+\varepsilon_7(\theta, c)\Big)\bigg)\right\rfloor.\end{equation}
This implies that
\begin{equation}\label{ENC}N_c=M_c(1)=\left\lfloor \frac{1}{c}\bigg(\Big({\int_\T \frac{dt}{\sqrt{s(t)}}}\Big)^2+\varepsilon_8(\theta, 1)\bigg)\right\rfloor\end{equation}
 and by Equations (\ref{neq 0}), (\ref{Emesure}), (\ref{EMC}) and (\ref{ENC}).
\begin{equation}\mu_c([0, \theta])=\frac{M_c(\theta)}{N_c}+\varepsilon_9(\theta, c)= \frac{\int_0^{\theta}\frac{dt}{\sqrt{s(t)}}}{\int_\T \frac{dt}{\sqrt{s(t)}}}+\varepsilon_{10}(\theta, c)=\mu_0([0,\theta])+\varepsilon_{11}(\theta, c).\end{equation}
As none of the measures $\mu_c$ has atoms, this implies that $c\mapsto \mu_c$ and all the maps $c\mapsto \mu_c([0, \theta])=h_c(\theta)$ are continuous. As every map $h_c$ is non decreasing in the variable $\theta$, we deduce from the Dini-Poly\`a Theorem  \cite[Exercise 13.b page 167]{Rudin} that $c\mapsto h_c$ is continuous for the $C^0$ uniform topology.
\end{proof}

\begin{remk}
If $\rho(c)=\frac{p}{q}$, then we proved that $\frac{\partial \eta_c(\theta)}{\partial c}=  \Big(\int_\T \frac{dt}{\sqrt{\textrm{$s_q\big(t, \eta_c(t)\big)$}}}\Big)^{-1}\frac{1}{\sqrt{\textrm{$s_q\big(\theta, \eta_c(\theta)\big)$}}}$ where $$Df^q(x)=\begin{pmatrix}
a_q(x)&s_q(x)\\
c_q(x)&d_q(x)
\end{pmatrix}.$$
This
 gives  for the conjugation
$$ h_c(\theta)=\mu_c([0, \theta])=\bigg(\int_\T \frac{dt}{\sqrt{s_q\big(t, \eta_c(t)\big)}}\bigg)^{-1}\int_0^\theta \frac{1}{\sqrt{s_q\big(t, \eta_c(t)\big)}}dt.$$
Observe that this $C^k$ depends on $\theta$.\\
Observe too that Equations (\ref{Eeta}) and (\ref{Eetaunif}) can be rewritten as
\begin{equation}\label{Eetab}\eta_c(\theta)=c\bigg[\bigg(\int_\T \frac{dt}{\sqrt{s_q\big(t, \eta_c(t)\big)}}\bigg)^{-1}\frac{1}{\sqrt{s_q\big(\theta, \eta_c(\theta)\big)}}+\varepsilon(\theta,c)\bigg],
\end{equation}
where
\begin{equation}\label{Eetaunifb}{\rm uniformly}\ \ {\rm for}\ \  \theta \in \T, \quad \lim_{c\rightarrow 0}\varepsilon (\theta, c)=0.\end{equation}

\end{remk}
\subsection{Generating function and regularity}  

 Let $u_c:\T\rightarrow \R$ be the $C^1$ function such that $u_c(0)=0$ and $\eta_{c}=c+u_c'$. In other words,  identifying $\T$ with $[0, 1]$, we have
 $$u_c(\theta)=\int_0^\theta \eta_{c}(t)dt-c\theta.$$
 Observe that for every $\theta$, the map $c\mapsto u_c(\theta)+c\theta$ is increasing because every $  c\mapsto \eta_c(\theta)$ is increasing.
  \begin{theorem}
 The map $(\theta, c)\mapsto u_c(\theta)$ is  $C^1$. 
 Moreover, in this case, $u$ is unique and  we have 
 \begin{itemize}
 \item the graph of $c+\frac{\partial u_c}{\partial \theta}$ is a leaf of the invariant foliation;
 \item $ \theta \mapsto \theta+\frac{\partial u_c}{\partial c}(\theta)$ is   the semi-conjugation $h_c$ between $g_c$ and $R_{\rho(c)}$ given in Theorem \ref{Tdiffrat}. We have: $h_c\circ g_c=h_c+\rho(c)$.
 
  \end{itemize}
 \end{theorem}

 \begin{cor}
 The semi-conjugation $h_c$      continuously depends on $c$. 
 \end{cor}

 \begin{proof}  The first point is a consequence of the definition of $u_c$.\\
 Then $u_c$ and $\frac{\partial u_c}{\partial \theta}=\eta_{c}-c$ continuously depend on $(\theta, c)$.\\
 Observe that with the  notation (\ref{E4}), we have  
 \begin{multline*}
 \Lambda_{c, c'} \big(\{ (\theta, r); \ \ \theta \in [\theta_1, \theta_2],  r\in [\eta_{c}(\theta), \eta_{c'}(\theta)]\}\big)= \\
 =\frac{1}{c'-c}\Big( \big(u_{c'}(\theta_2)-u_{c'}(\theta_1)\big)-\big(u_{c}(\theta_2)-u_{c}(\theta_1)\big) \Big)+(\theta_2-\theta_1).
 \end{multline*}
  Moreover, if $\rho(c_0)\in \Q$, we deduce from   Lemma \ref{Lratmeas} that $u_c$   admits a derivative with respect to $c$ at $c_0$   $$\frac{\partial u_c}{\partial c}_{|c=c_0}(\theta)=\lim_{c\rightarrow c_0}\frac{1}{c-c_0}\Big(\big(u_{c}(\theta)-u_{c}(0)\big)-\big(u_{c_0}(\theta)-u_{c_0}(0)\big)\Big)$$
 that is given by

   $$\frac{\partial u_c}{\partial c}_{|c=c_0}(\theta)=\mu_{c_0}([0, \theta])-\theta=h_{c_0}(\theta)-\theta$$
  and this derivative continuously depends on $\theta$.\\
Assume now that $\rho(c_0)$ is   irrational and let $c$ tend to $c_0$. Every limit point of $\Lambda_{c,  c_0}$ when $c$ tends to $c_0$  is a Borel probability measure that is invariant by $f$ and supported on the graph of $\eta_{ c_0}$. As there exists only one such measure,  whose projection was denoted by $\mu_{c_0}$,  we deduce that 
$${ \pi_{1*}}\Big(\lim_{c\rightarrow c_0}\Lambda_{c,  c_0}\Big) =\mu_{c_0}.$$
 As $\mu_{c_0}$ has no atom, we have for all $\theta_0\in [0, 1)$
\begin{multline*}
  h_{c_0}(\theta_0)=\ \mu_{c_0}([0, \theta_0])  \\
\!\!\!\!\!\!\!\!\!\!\!\!\!\!=\lim_{c\rightarrow c_0}\Lambda_{c_0, c} (\{ (\theta, r); \theta \in [0, \theta_0],  r\in [\eta_{c_0}(\theta), \eta_{c}(\theta)]\}) \\
\qquad \quad =\lim_{c\rightarrow c_0}\frac{1}{c-c_0}\Big(\big(u_{c}(\theta_0)-u_{c}(0)\big)-\big(u_{c_0}(\theta_0)-u_{c_0}(0)\big)\Big)+\theta_0 \\
=\frac{\partial u_c}{\partial c}(\theta_0)_{|c=c_0}+\theta_{0},
\end{multline*}
hence $u_c$ admits a derivative with respect to $c$  and 
$$h_{c_0}(\theta)=\mu_{c_0}([0, \theta])=\theta+\frac{\partial u_c}{\partial c}(\theta)_{|c=c_0}.$$

Because of Theorem \ref{Tdiffrat}, $(\theta, c)\mapsto \frac{\partial u_c}{\partial c}(\theta)=h_c(\theta)-\theta$ is continuous. As the two partial derivatives $\frac{\partial u_c}{\partial \theta}$ and $\frac{\partial u_c}{\partial c}$ are continuous in $(\theta, c)$, we conclude that $u$ is $C^1$.

\end{proof}

\section{Proof of the implication (\ref{pt2Tgeneder}) $\Rightarrow$  (\ref{pt1Tgeneder})  in   Theorem \ref{Tgeneder}}\label{secsecond}
We use the same notations as in Theorem \ref{Tgenecont}. We assume that the map $u$ is   $C^1$. \\
Then the graph of every $\eta_c=c+\frac{\partial u_c}{\partial \theta}$ is a  continuous graph that is backward invariant, hence invariant.
If for $c_1\not=c_2$ the two graphs of $\eta_{c_1}$ and $\eta_{c_2}$ have a non-empty intersection, then their common rotation number is rational because a symplectic twist map has at most one invariant curve with a fixed irrational rotation number (see \cite{He1}).   Moreover, for every $c\in [c_1, c_2]$, we have $\rho(c)=\rho(c_1)$.\\
 Using results   of \cite{Ban} (see section 5),  we know that above any $\theta\in \T$, there are at most two $r_1, r_2\in\R$ such that the orbit of $(\theta, r_i)$ is minimizing with rotation number $\rho(c_1)$. As $c_1\not=c_2$, there exists then $\theta\in \T$ such that $r_1=\eta_{c_1}(\theta)\not= \eta_{c_2}(\theta)=r_2$. But for $c\in [c_1, c_2]$, the orbit of $\big(\theta, \eta_c(\theta)\big)$ is minimizing with rotation number equal to $\rho(c_1)$ and then $\eta_c(\theta)\in \{ r_1, r_2\}$. As $c\mapsto \eta_c(\theta)$ is continuous with values in $\{\eta_{c_1}(\theta), \eta_{c_2}(\theta)\}$ and satisfies $\eta_{c_1}(\theta)\not=\eta_{c_2}(\theta)$, we obtain a contradiction.\\
 
   So finally the graphs of the $\eta_c$ define a lamination of $\A$ and then $f$ is $C^0$-integrable.

\section{Proof of Theorem \ref{corLip}}\label{secth13}

\begin{proposition}\label{rho}
Assume that the $C^1$ symplectic twist map $f:\A\rightarrow \A$ has an invariant  locally Lipschitz continuous  foliation into graphs $c\in\R\mapsto  \eta_c\in C^{0}(\T, \R)$. Then the map $\rho: c\in\R\mapsto \rho(c)$ is a locally biLipschitz homeomorphism.
\end{proposition}
 This result will not be used in what follows and its proof is postponed to the end of this section.

\subsection{Proof of the first implication}\label{ss1impl}
We assume that the invariant foliation is   $K$-Lipschitz on a compact $\ck =\{ (\theta, \eta_c(\theta)); \theta\in \T, c\in [a, b]\}$, which means 
\begin{equation}\label{Efollip}
\forall \theta \in \T, \forall c_1, c_2\in  [a,b], \quad \frac{|c_1-c_2|}{K}\leq \left| \eta_{c_1}(\theta)-\eta_{c_2}(\theta)\right|\leq K|c_1-c_2|.
\end{equation} 
 As the Lispchitz constant of the invariant graphs are locally uniform in $c$, changing $\ck$ and $K$, we also have
\begin{equation*}
\forall \theta_1, \theta_2 \in \R, \forall c\in [a,b],\quad  \left| \eta_{c}(\theta_1)-\eta_{c}(\theta_2)\right|\leq K|\theta_1-\theta_2|.
\end{equation*} 
and then 
\begin{equation*}
\forall \theta_1, \theta_2 \in \R, \forall c_1, c_2\in [a,b],\quad   \left| \eta_{c_1}(\theta_1)-\eta_{c_2}(\theta_2)\right|\leq K\left(|\theta_1-\theta_2|+|c_1-c_2]\right).
\end{equation*}

 Hence the map  $(\theta, c)\mapsto \eta_{c}(\theta)$ is Lipschitz and then Lebesgue almost everywhere differentiable by Rademacher theorem. We denote the set of its differentiability points in $\T\times [a, b]$ by $\cN$. Let us fix some $(\theta_0, c_0)\in \T\times \R$    where $(\theta, c)\mapsto \eta_{c}(\theta)$ is differentiable.  Because of Equation (\ref{Efollip}), we have $\frac{\partial \eta}{\partial c}(\theta_0, c_0)\geq \frac{1}{K}$.\\
 Along the orbit $\big(\theta_k, \eta_{c_0}(\theta_k)\big)$ of $\big(\theta_0, \eta_{c_0}(\theta_0)\big)$, we use the basis $\big(1, \eta_{c_0}'(\theta_k)\big)$ of the tangent subspace. We have in the  basis $\Big( \big(1, \eta_{c_0}'(\theta_j)\big), (0, 1)\Big)_{j\in\Z}$ the following symplectic matrix 
$$Df^k\big(\theta_0, \eta_{c_0}(\theta_0)\big)=\begin{pmatrix} a_k&b_k\\ 0 & d_k \end{pmatrix}$$
where $a_k=\frac{\partial g_{c_0}^k}{\partial \theta}(\theta_0)$.
\\
We recall that $g_c(\theta)=\pi_1\circ f\big(\theta, \eta_c(\theta)\big)$ and that   
\begin{equation}\label{Eitlip} \forall c\in \R,\forall k\in\Z,  \forall \theta\in \T, \quad f^k\big(\theta, \eta_c(\theta)\big)=\big(g_c^k(\theta), \eta_c\circ g_c^k(\theta)\big).\end{equation}
 Observe that this implies that $g_c^k(\theta)=\pi_1\circ f^k\big(\theta, \eta_c(\theta)\big)$. Moreover, using the fact that
 $f^k$ is a twist map (see Proposition \ref{iterateTwist}),
   we also deduce $b_k>0$ and
 $$\frac{\partial g_{c_0}^k}{\partial c}(\theta_0)=b_k\frac{\partial \eta_{c_0}}{\partial c}(\theta_0).$$
 
Equation (\ref{Eitlip})  implies that {the functions $(\theta,c)\mapsto g^k_{c}(\theta)$ are differentiable at $(\theta_0,c_0)$ and }
$$Df^k\big(\theta_0, \eta_{c_0}(\theta_0)\big)\Big(0, \frac{\partial\eta_{c_0}}{\partial c}(\theta_0)\Big)=\Big(\frac{\partial g^k_{c_0}}{\partial c}(\theta_0), \frac{\partial\eta_{c_0}}{\partial c}\big(g_{c_0}^k(\theta_0)\big)+\eta_{c_0}'\big(g_{c_0}^k(\theta_0)\big)\frac{\partial g^k_{c_0}}{\partial c}(\theta_0)\Big),
$$
i.e.
  $$Df^k\big(\theta, \eta_{c_0}(\theta_0)\big)\Big(0, \frac{\partial\eta_{c_0}}{\partial c}(\theta_0)\Big)=\frac{\partial g^k_{c_0}}{\partial c}(\theta_0)\Big(1, \eta_c'\big(g_{c_0}^k(\theta_0)\big)\Big)+\frac{\partial \eta_{c_0}}{\partial c}\big(g_{c_0}^k(\theta_0)\big)(0,1),
$$

i.e.
$$ b_k\frac{\partial\eta_{c_0}}{\partial c}(\theta_0)=\frac{\partial g^k_{c_0}}{\partial c}(\theta_0)\quad{\rm and}\quad d_k= \frac{\partial\eta_{c_0}}{\partial c}\big(g_{c_0}^k(\theta_0)\big) \left(\frac{\partial\eta_{c_0}}{\partial c}(\theta_0)\right)^{-1}.$$
The matrix being symplectic,   we have $a_kd_k=1$ and then 
$$\frac{\partial g_{c_0}^k}{\partial \theta}(\theta_0)= \frac{\partial\eta_{c_0}}{\partial c}(\theta_0)\left(  \frac{\partial\eta_{c_0}}{\partial c}\big(g_{c_0}^k(\theta_0)\big)\right)^{-1}
\in [\frac{1}{K^2}, K^2]
$$ is uniformly bounded.\\
 As $\cN$ has full Lebesgue measure in $\T\times [a, b]$, there exists a set $C\subset [a, b]$ with full Lebesgue measure such that for every $c_0\in C$, $\cN\cap (\T\times \{ c_0\})$ has full Lebesgue measure in $\T\times \{ c_0\}$. We obtain that for every $c\in C$, the familly $(g_c^k)_{k\in\Z}$ is uniformly $K^2$-Lipschitz. As $C$ is dense in $[a, b]$, we deduce by continuity of $c\mapsto g_c$ that the maps $\{g_c^k, \ \ k\in\Z,\  c\in [a,b]\}$ are $K^2$ Lipschitz.
\\
Finally,  every $g_c$ is a biLipschitz orientation preserving homeomorphism of $\T$ whose all iterated homeomorphisms are equiLipschitz. We deduce  from results of \cite{Arna1} that $\eta_c$ is in fact $C^1$ (and the two Green bundles coincide along its graphs) and that $g_c$ is $C^1$ conjugated to a rotation.
Hence all the points are recurrent. Moreover, as the two Green bundles are equal everywhere, they are continuous. Because they coincide with the tangent space to the foliation, the foliation is a $C^1$ lamination.\\
As the $(g_c^k)'$ are equibounded by some constant $\widetilde K$, we deduce from results that are contained in \cite{Her2} that the conjugations $h_c$ to a rotation are $\widetilde K$-equibiLipschitz. \\
Finally, we deduce from Theorem \ref{Tgeneder} that $u$ is $C^1$  with partial derivatives that are
\begin{itemize}
\item $\frac{\partial u_c}{\partial\theta}(\theta)=\eta_c(\theta)-c$ which is  locally Lipschitz (as $\eta_c$ is)   as a function of $(\theta,c)$;
\item $\frac{\partial u_c}{\partial c}(\theta)=h_c(\theta)-\theta$ which is uniformly Lipschitz\footnote{This function is even $C^1$ at $c$'s such that $\rho(c)$ is irrational.} in the variable $\theta$ on any compact set of $c$'s.
\end{itemize}
 If we denote by $K$ a local Lipschitz constant for $h_c$ and $h_c^{-1}$, we have Lebesgue almost everywhere that
$$\frac{\partial h_c}{\partial \theta}(\theta)\in \Big[\frac{1}{K}, K\Big]$$
and then
$$\frac{\partial^2 u_c}{\partial \theta\partial c}(\theta)\in \Big[-1+\frac{1}{K}, -1+K\Big],$$ that gives the last point of Theorem \ref{corLip}.\\
 Note that this improves the fact that $u$ is $C^1$.

Let $v:\R^2\rightarrow \R_+$ be the $C^\infty$ function with support in $B(0, 1)$ defined  by $v(\theta, c) = a \exp\big((1-\|(\theta, c)\|)^{-2}\big)$ for $(\theta, c) \in B(0,1)$ and where $a$ is    such that $\int  v=1$. We denote by $v_\varepsilon$ the function $v_\varepsilon (x)=\frac{1}{\varepsilon^2}v(\frac{x}{\varepsilon})$. Then we define for every $\varepsilon>0$.
$$U_\varepsilon(\theta, c)=(u*v_\varepsilon) (\theta, c),$$
where we recall the formula for the convolution
$$u*v(x)= \int u(x-y)v(y)dy.$$
 Then $U_\varepsilon$ is  $1$-periodic in $\theta$  and smooth and when $\varepsilon$ tends to $0$, the functions $U_\varepsilon$ tend  to $U$ in the $C^1$ compact-open  topology.\\
Observe that for every $\theta$, the function $c\mapsto c+\frac{\partial u}{\partial \theta}(\theta,c)$ is  increasing.
 We deduce that  the convolution $c\mapsto c+\frac{\partial U_\varepsilon}{\partial \theta}(\theta,c)$ is a $C^\infty$ diffeomorphism as it is a mean of $C^\infty$ diffeomorphisms  thanks to Lemma \ref{Philippe}. Finally, the maps $F_\varepsilon: (\theta, c)\mapsto \big(\theta,  c+\frac{\partial U_\varepsilon}{\partial \theta}(\theta, c)\big)$ define  $C^\infty$ foliations that  converge to the initial foliation $F_0: (\theta, c)\mapsto \big(\theta,  c+\frac{\partial u}{\partial \theta}(\theta, c)\big)$ for the $C^0$ compact-open topology when $\varepsilon$ tends to $0$.\\ 
Observe that the  $h_c$'s are   assumed to be increasing. We deduce  that the  maps $G_\varepsilon: (\theta, c)\mapsto (\theta+\frac{\partial U_\varepsilon}{\partial c}(\theta,c), c)$ are $C^\infty$ diffeomorphisms of $\T\times \R$ that    converge for the $C^0$  compact-open topology   to  $G_0:(\theta, c)\mapsto (\theta+\frac{\partial u}{\partial c}(\theta, c), c)$.\\
Finally, the $\ch_\varepsilon= G_\varepsilon\circ F_\varepsilon^{-1}$  are $C^\infty$ diffeomorphisms of $\T\times \R$ that converge   for the $C^0$  compact-open topology  to $G_0\circ F_0^{-1}=\Phi$.

This exactly means that $\Phi$ is a symplectic homeomorphism. Moreover, we have
$$\Phi\circ f\circ \Phi^{-1}(x, c)= G_0\circ F_0^{-1}\circ F\circ F_0\circ G_0^{-1}(x, c)= (x+\rho(c), c).$$

\begin{lemma}\label{Philippe}
Let $f : \R\to \R$ be a non-negative, non-trivial, smooth, integrable and even function such that $ f' \leq 0$ on $[0,+\infty)$. Then if $g : \R \to \R$ is increasing, $f*g$ is an increasing $C^\infty$ diffeomorphism.
\end{lemma}
\begin{proof}
As $f$ is even, $f'$ is odd. Just notice that 
$$(f*g)'(x) = \int_\R f'(y)g(x-y) dy = \int_0^{+\infty} f'(y) \big(g(x-y)-g(x+y)\big) dy.$$
The result follows as $g(x-y)-g(x+y) <0$ and $f'(y) \leq 0$ and does not vanish everywhere.
\end{proof}

We conclude this section by returning to the proof of Proposition \ref{rho}. We will use the following
\begin{lemma}\label{lemme rho}
Let $f,g : \R \to \R$ be lifts of homeomorphisms of $\T$ that preserve orientation (implying $f(\cdot +1) = f(\cdot) +1$ and $g(\cdot +1) = g(\cdot) +1$). Assume that 
\begin{itemize}
\item either $f$ or $g$ is conjugated to a translation $t_\alpha : x \mapsto x+\alpha$ by a homeomorphism $h$ lift of a homeomorphism of $\T$ that preserves orientation;
\item $h$ and $h^{-1}$ are $K$-Lipschitz.
\end{itemize}

\begin{enumerate}
\item If there exists $d>0$ such that $f<g+d$, then
 $\rho(f)\leq \rho(g)+Kd$.
 \item If there exists $d>0$ such that $f+d <g$ 
 then $\rho(f) +\frac{d}{K}\leq \rho(g)$.
 \end{enumerate}
\end{lemma}
\begin{proof}
Let us say that $h\circ g \circ h^{-1} = t_\alpha$, hence $\rho(g) = \alpha$ (the proof when $f$ is conjugated to a translation is the same). 
\begin{enumerate}
\item 
By hypothesis, $ f\circ h^{-1}< g\circ h^{-1} +d $. Using that $h$ is increasing and  $K$-Lipschitz, it follows that for all $ x\in \R$,
$$ \quad h\circ f\circ h^{-1}(x)<h(g\circ h^{-1}(x) +d)< h\circ g\circ h^{-1}(x)+Kd=x+\alpha+Kd.$$
Finally, as $\rho(f) = \rho(  h\circ f\circ h^{-1})$, we conclude that 
$$\rho(f)\leq \alpha+Kd = \rho(g)+Kd.$$
\item By hypothesis, $ f\circ h^{-1}+d< g\circ h^{-1}  $. Using that $h$ is increasing, it follows that
$$\forall x\in \R, \quad h( f\circ h^{-1}(x)+d)<h\circ g\circ h^{-1}(x)=x+\alpha.$$
Because $h^{-1}$ is $K$-Lipschitz and increasing, observe that
\begin{multline*}
d=h^{-1}(h(f\circ h^{-1}(x)+d))-h^{-1}(h\circ f\circ h^{-1}(x))\\
\leq K\left( h(f\circ h^{-1}(x)+d)-h\circ f\circ h^{-1}(x)\right).
\end{multline*}
Then
$$h\circ f\circ h^{-1}(x)\leq h(f\circ h^{-1}(x)+d)-\frac{d}{K}<x+\alpha -\frac{d}{K};$$
hence $\rho(f)+\frac{d}{K}\leq \rho(g)$.

\end{enumerate}
\end{proof}

\begin{proof}[Proof of Proposition \ref{rho}]
The proof is now a direct application of the previous Lemma. Indeed, we have seen that when the foliation is $K$-Lipschitz, if $c$ varies in a compact  set $\mathcal K$, the dynamics $g_c$ are all conjugated to rotations. We have moreover proven there exists a constant $\widetilde K$ such that the conjugating functions $h_c$ may be chosen equi-Lipschitz (for $c\in \mathcal K$).
Finally, when $\rho(c)$ is irrational, we deduce from results of \cite{Her2}  (see also \cite{ArnaXue}) that $h_c^{-1}$ is also $\widetilde K$-Lipschitz. We therefore conclude that $\rho$ is $K\widetilde K$-Lipschitz when restricted to $\rho^{-1}(\R\setminus \Q)$. By density, $\rho$ is Lipschitz.

 We denote the minimum torsion on $\ck$ by
$$b_{\rm min}=\min_{x\in\ck}\frac{\partial f_1}{\partial \theta}(x).$$
For $c_1<c_2$ in $[a, b]$ such that either $\rho(c_1)$ or $\rho(c_2)$ is irrational, we have
\begin{multline*}
\tilde g_{c_2}(\theta)-\tilde g_{c_1}(\theta)=F_1\big(\theta, \eta_{c_2}(\theta)\big)-F_1\big(\theta, \eta_{c_1}(\theta)\big)\\
\geq b_{\rm min}\big(\eta_{c_2}(\theta)-\eta_{c_1}(\theta)\big)\geq \frac{b_{\rm min}}{K}(c_2-c_1).
\end{multline*}
We deduce from the second point of Lemma \ref{lemme rho} that
$$\rho(g_{c_2})-\rho(g_{c_1})\geq \frac{b_{\rm min}}{K^2}(c_2-c_1).$$
As previously, by density, we get that $\rho^{-1}$ is also locally Lipschitz.  
\end{proof}

\subsection{Proof of the second implication}
We assume that  the map $u$ is $C^1$ with $\frac{\partial u}{\partial \theta}$ locally Lipschitz continuous and $\frac{\partial u}{\partial c}$ uniformly Lipschitz in the variable  $\theta$ on any compact set of $c$'s and there exists a constant $k>-1$ such that $\frac{\partial^2 u}{\partial \theta\partial c}(\theta,c)>k$ almost everywhere.\\
  Theorem \ref{Tgeneder} yields that
  the graphs of the $\eta_c$ define a lamination of $\A$ into Lipschitz graphs and that the map   $h_c:\theta \mapsto \theta+\frac{\partial u_c}{\partial c}(\theta)$ is a semi-conjugation between the projected Dynamics $g_c: \theta\mapsto \pi_1\circ f\big(\theta, c+\frac{\partial u_c}{\partial \theta}(\theta)\big)$ and a rotation $R$ of $\T$, i.e. $h_{c}\circ g_c=R\circ h_{c}.$
 
By assumption, $\eta_c=c+\frac{\partial u_c}{\partial \theta}$ is locally Lipschitz. We want to prove that $(\theta, c)\mapsto \big(\theta, \eta_c(\theta)\big)$ is locally biLipschitz. We only need to prove that locally, we have for Lebesgue almost every $(\theta, c)$ a uniform positive lower bound for $\frac{\partial \eta_c}{\partial c}$ (observe that $\frac{\partial \eta_c}{\partial c}$ is always non-negative because every $c\mapsto \eta_c(\theta)$ is increasing). \\
 As $\frac{\partial^2 u}{\partial \theta\partial c}(\theta,c)>k$ almost everywhere, the set 
$$C=\Big\{ c\in \R; \quad \frac{\partial^2 u}{\partial \theta\partial c}(\theta,c)>k\quad{\rm for}\quad{\rm Lebesgue} \quad{\rm almost}\quad{\rm every}\quad \theta\in\T\Big\}$$
has full Lebesgue measure.

Then, for every $c\in C$ and $\theta>\theta'$ in $[0, 1]$, we have
$$\frac{\partial u}{\partial c}(\theta, c)-\frac{\partial u}{\partial c}(\theta', c)=\int_{\theta'}^{\theta} \frac{\partial^2 u}{\partial \theta\partial c}(a,c)da\geq k(\theta-\theta').$$

Hence, if  $c>c'$ in $\R$, we have
$$u(\theta, c)-u(\theta, c')-u(\theta', c)+u(\theta', c')= \int_{c'}^c\left(  \frac{\partial u}{\partial c}(\theta, t)-\frac{\partial u}{\partial c}(\theta', t)  \right)dt\geq   k(\theta-\theta')(c-c').$$
If we divide by $\theta-\theta'$ and take the limit $\theta'\rightarrow \theta$, we obtain
$$\frac{\partial u}{\partial \theta}(\theta, c)-\frac{\partial u}{\partial \theta}(\theta, c')\geq k(c-c'),$$
that is equivalent to 
$$\forall \theta\in \T, \quad \eta_c(\theta)-\eta_{c'}(\theta)\geq (1+k)(c-c').$$
As $1+k>0$, we conclude that the foliation is biLipschitz.

\section{Foliations by  graphs}\label{Sfollag}
\subsection{Proof of Proposition \ref{Psymfolarn}}
 Let $f:\A \to \A$ be an exact  symplectic homeomorphism.  We assume the $f$ invariant foliation $\cF$ into $C^0$ graphs  is  symplectically homeomorphic (by $\Phi : \A \to \A$) to the standard foliation $\cF_0=\Phi(\cF)$. Then the standard foliation is invariant by the exact symplectic homeomorphism $g=\Phi\circ f\circ \Phi^{-1}$. Hence we have
 $$g(\theta, r)=(g_1(\theta ,r), r).$$
 As $g$ is area preserving, for every $\theta\in [0, 1]$ and every $r_1<r_2$, the area of $[0, \theta]\times [r_1, r_2]$ is equal to the area of $g\big( [0, \theta]\times [r_1, r_2]\big)$, i.e.
 $$\theta(r_2-r_1)=\int_{r_1}^{r_2} \big(g_1(\theta, r)-g_1(0, r)\big)dr.$$
 Dividing by $r_2-r_1$ and taking the limit  when $r_2$ tends to $r_1$, we obtain
 $$g_1(\theta, r_1)=\theta+g(0, r_1).$$
 This proves the proposition for $\rho=g_1(0, \cdot)$.
 
 \subsection{Proof of Theorem \ref{TC0arn}}
Let us consider a $C^0$-foliation $\cF$ of $\A$: $(\theta, c) \mapsto \big(\theta , \eta_c(\theta)\big)$, where $\int_\T \eta_c = c$. Then there exists a continuous function $u:\A\rightarrow \R$ that admits a continuous derivative with respect to $\theta$ such that $\eta_c(\theta)=c+\frac{\partial u}{\partial \theta}(\theta,c)$ and $u(0, c)=0$. \\
{\bf Proof of the first implication.}\\
We assume that this foliation is  exact symplectically homeomorphic to the standard foliation $\cF_0=\Phi(\cF)$ by some exact symplectic homeomorphism $\Phi$.

Observe that the foliation $\cF$ is transverse to the ``vertical'' foliation $\cg_0$ into $\{\theta\}\times \R$ for $\theta\in\T$. Hence the foliation $\cg=\Phi(\cg_0)$ is a foliation that is transverse to the standard (``horizontal'') foliation $\cF_0=\Phi(\cF)$. This exactly means that the foliation  $\cg$ is a foliation into graphs of maps $\zeta_\theta: \R\rightarrow \T$. Hence there exists a continuous function $v:\A\rightarrow \R$ that admits a continuous derivative with respect to $r$ such that the foliation $\cg$ is the foliation into graphs $\Phi(\{\theta\}\times\R)$ of $\zeta_\theta: r\mapsto \theta+\frac{\partial v}{\partial r}(\theta, r)$. Observe that by definition of $\zeta_\theta$, we have $\Phi\big(\zeta_\theta(c), c\big)=\big(\theta,\eta_c(\theta)\big)$. As a result, every map $\theta\mapsto \zeta_\theta(c)$ is a homeomorphism of $\T$.
\begin{center}
\includegraphics[width=8cm]{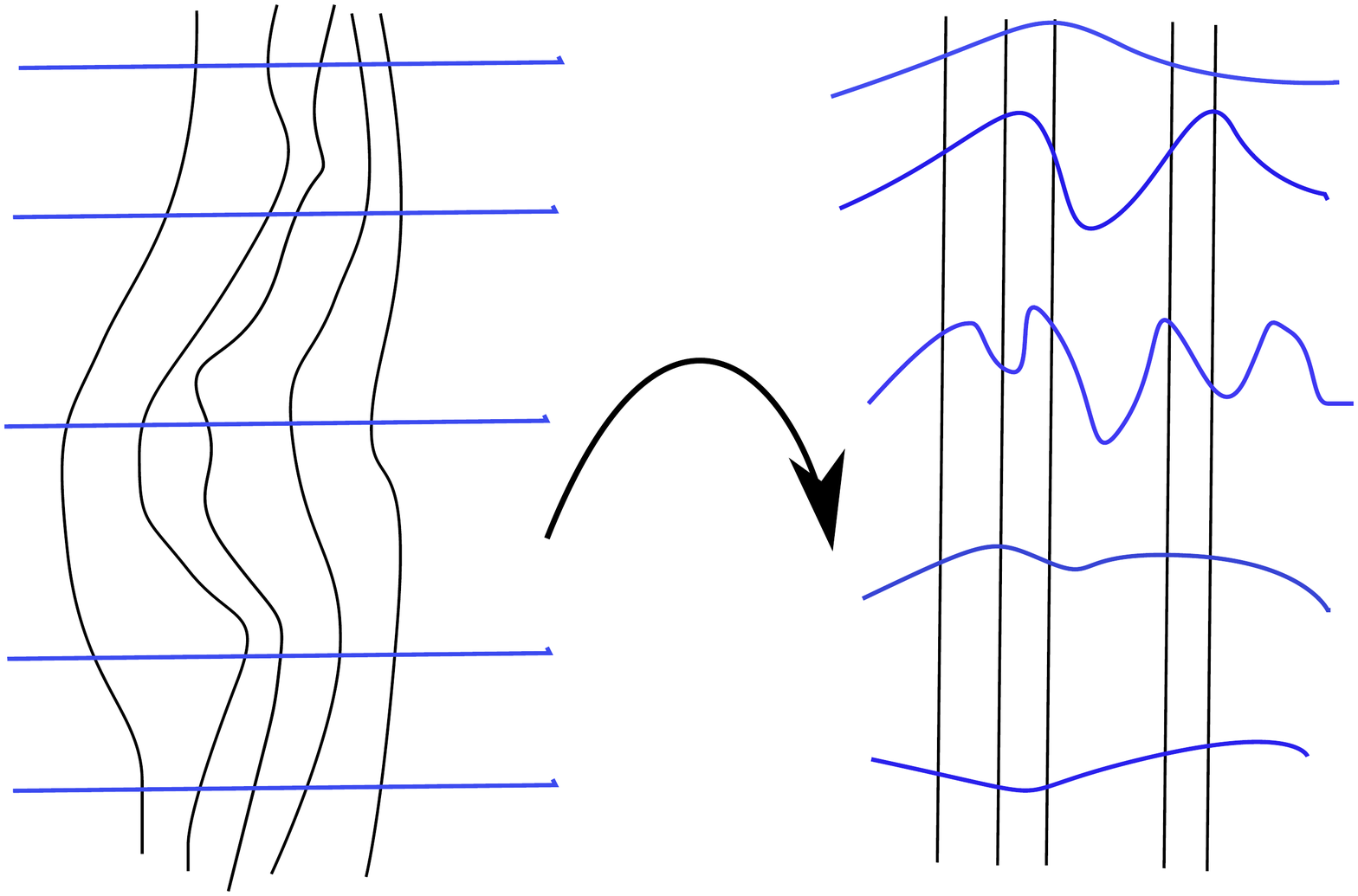}
\end{center}

We now use the preservation of the area. We fix $\theta_1<\theta_2$ in $[0, 1]$ and $r_1<r_2$ in $\R$. Because $\Phi$ is a symplectic homeomorphism, $\Phi$ preserves the area and so the two following domains have the same area
\begin{itemize}
\item the domain delimited by $\T\times \{ c_1\}$, $\T\times \{ c_2\}$, the graph of $c\in\R\mapsto \zeta_{\theta_1}(c)$ and the graph of $c\in\R\mapsto \zeta_{\theta_2}(c)$;
\item the domain delimited by the graphs of $\eta_{c_1}$, $\eta_{c_2}$ and the verticals $\{\theta_1\}\times\R$ and $\{ \theta_2\}\times \R$.
\end{itemize}
This can be written
$$  \int_{c_1}^{c_2}\Big( \big(\theta_2+\frac{\partial v}{\partial c} (\theta_2,c)\big)-\big(\theta_1+\frac{\partial v}{\partial c} (\theta_1,c)\big)\Big) dc=\int_{\theta_1}^{\theta_2}\Big(\big(c_2+\frac{\partial u}{\partial \theta} (\theta,c_2)\big) - \big(c_1+\frac{\partial u}{\partial \theta} (\theta,c_1)\big)\Big)d\theta .$$
It follows that 
\begin{multline*}
u(\theta_2,c_2)-u(\theta_1,c_2)-u(\theta_2,c_1)+u(\theta_1,c_1) = \\
=v(\theta_2,c_2) -v(\theta_2,c_1) -v(\theta_1,c_2) +v(\theta_1,c_1).
\end{multline*}

Evaluating for $\theta_1=0$ we find
$$u(\theta_2,c_2)-u(\theta_2,c_1) = v(\theta_2,c_2) -v(\theta_2,c_1) -v(0,c_2) +v(0,c_1).$$
Finally, as $v$ admits a continuous partial derivative with respect to $c$, we conclude that 
$\frac{\partial u}{\partial c}(\theta, c) = \frac{\partial v}{\partial c}(\theta, c)-\frac{\partial v}{\partial c}(0, c)$ exists and is continuous. Hence $u$ is $C^1$.  Moreover, every map $\theta\mapsto \theta+\frac{\partial u}{\partial c}(\theta, c)= \zeta_c(\theta)-\frac{\partial v}{\partial c}(0, c)$ is a homeomorphism of $\T$ and we have established the first implication.\\
{\bf Proof of the second implication.}\\
We assume that there exists a $C^1$ map $u : \A \to \R$ such that 
\begin{itemize}
\item $u(0,r) = 0$ for all $r\in \R$,
\item $\eta_c(\theta) = c+ \frac{\partial u}{\partial \theta}(\theta,c)$ for all $(\theta, c) \in \A$,
\item for all $c\in \R$, the map $\theta \mapsto \theta + \frac{\partial u}{\partial c}(\theta , c)$ is a homeomorphism of $\T$.
\end{itemize}
Then we can define a unique homeomorphism $\Phi$ of $\A$ by  $$\Phi\Big(\theta + \frac{\partial u}{\partial c}(\theta , c), c\Big)=\Big(\theta, c+ \frac{\partial u}{\partial \theta}(\theta,c)\Big).$$
The previous computations (with $v=u$) proves that $\Phi$ preserves the area and so is an exact symplectic homeomorphism.
\subsection{Proof of Corollary \ref{Corota}}

The  {\em   if} part is obvious.

Let us prove the {\em  only if} part, that is we assume $f$ is $C^0$-integrable with the Dynamics on each leaf conjugated to a rotation. We denote by    $u : \A \to \R$ the map given by  theorem \ref{Tgenecont} and that enjoys the properties of Theorems \ref{Tgeneder}  and \ref{Tdiffrat}.  Hence $h_c:\theta \mapsto \theta+\frac{\partial u_c}{\partial c}(\theta)$ is a semi-conjugation between the projected Dynamics $g_c: \theta\mapsto \pi_1\circ f\big(\theta, c+\frac{\partial u_c}{\partial \theta}(\theta)\big)$ and the rotation $R_{\rho(c)}$ of $\T$ and even is a conjugation when $\rho(c)$ is rational.\\
If $\rho(c)$ is irrational, it follows from the hypothesis that $g_c$ is conjugated to a rotation. As the dynamics is minimal, there is up to constants a unique (semi)-conjugacy and and then $h_c$ is a true conjugation.

As an application, here is a Lipschitz foliation that is not symplectically homeomorphic to the standard foliation. Let $\eta_c(\theta) =  c+ \varepsilon(c)\cos(2\pi \theta)$. We assume that $\varepsilon$ is a contraction ($k$-Lipschitz with $k<1$) that is not everywhere differentiable. It follows that $(\theta,c)\mapsto \eta_c(\theta)$ is a biLipschitz foliation. Were it   symplectically homeomorphic to the standard foliation,   the associated function given by Theorem  \ref{TC0arn} would be 
$$(\theta,c)\mapsto u_c(\theta)=\frac{\varepsilon(c)}{2\pi}\sin(2\pi \theta).$$
However, by Theorem \ref{TC0arn}, this function should be $C^1$ which is not the case as it does not admit partial derivatives with respect to $c$.

 \appendix
 
 \section{Examples}\label{AppA}
 \subsection{An example a semi-concave function that is not a weak K.A.M. solution for $\widehat T^c$ and that satisfies $f^{-1}\big(\overline{\cg(c+u')}\big)\subset \cg(c+u')$}\label{sspseudo}
 Let us begin by introducing $g_t:\A\rightarrow\A$ as being the time $t$ map  of the Hamiltonian flow of the double pendulum Hamiltonian
 $$H(\theta, r)=\frac{1}{2}r^2+\cos (4\pi\theta).$$
 If  $t>0$ is small enough, $g_t$ is a twist map.\\
 Observe that $H$ is a so-called Tonelli Hamiltonian (see \cite{Fa3} for the definition) with associated Lagrangian $L(\theta, v)=\frac{1}{2}v^2-\cos (4\pi\theta)$. The global minimum $-1$ of $L$ is attained in $(0, 0)$ and $(\frac{1}{2}, 0)$.\\
 If $G_t$ is the time $t$ map of the lift of $H$ to $\R^2$, then $G_t$ is a lift of $g_t$ and if $G_s(\theta, r)=(\theta_s, r_s)$, a generating function of $G_t$ is 
 $$S_t(\theta, \theta_t)=\int_0^tL(\theta_s, \dot\theta_s)ds.$$
 By using this formula, observe that the only ergodic minimizing measures for the cohomology class $0$ are the Dirac measure at $0$ and $\frac{1}{2}$.\\

 Then we denote by $h:\A\rightarrow \A$ the map that is defined by $h(\theta, r)=(\theta+\frac{1}{2}, r)$. Then $f=h\circ g_t=g_t\circ h$  is again a twist map and $H$ is an  integral for $f$, which means that $H\circ f=H$.\\
 It is easy to check that a generating function of a lift $F$ of $f$ is given by
 $$S(\theta, \Theta)= S_t\big(\theta, \Theta-\frac{1}{2}\big).$$
 From this, we deduce that the Mather set corresponding to the cohomology class zero (and the rotation number $\frac{1}{2}$) is the support of a unique ergodic measure, that is the mean of two Dirac measure $\frac{1}{2}(\delta_{(0, 0)}+\delta_{(\frac{1}{2}, 0)})$.\\
 As there is only one such minimizing measure, we know that there is a unique, up to constants, weak K.A.M. solution $u$ with cohomology class $0$. But there are a lot of graphs of $v'$ with $v:\T\rightarrow \R$ semi-concave that are invariant by $f$. The first one we draw corresponds to the weak K.A.M. solution whose graph is strictly mapped into itself by   $f^{-1}$. Perturbing slightly the pseudograph in the level $\{ H=1\}$, we obtain another backward invariant pseudograph that doesn't correspond to a weak K.A.M. solution.
 
 In the right drawing, the perturbation of the pseudograph must be small enough so that, in the right eye on the upper manifold, the piece of pseudograph that goes beyond the vertical dotted line is mapped  by $f$ in the upper piece of pseudograph of the left eye. \begin{center}
\includegraphics[width=10cm]{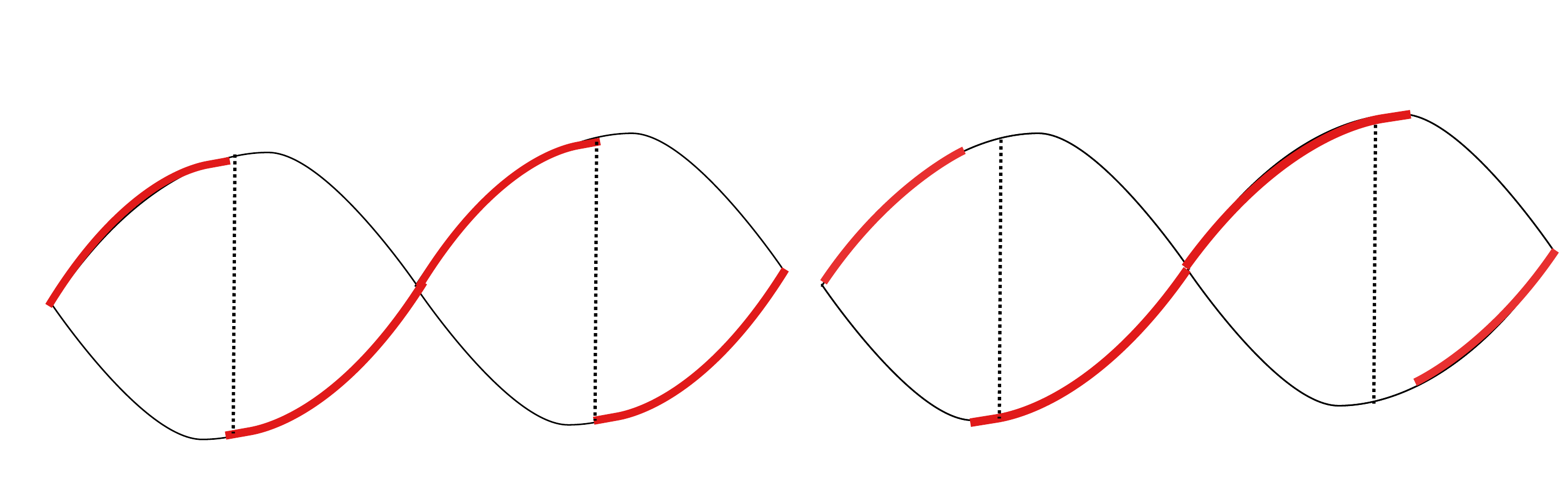}
\end{center}

 \subsection{Cases where the discounted solution doesn't depend continuously on $c$}\label{sscomplint}
 
 Let us start this appendix of counterexamples with a positive result. We will show that even if discounted solutions may depend in a discontinuous way on $c$, the same is not true for their derivative. In what follows we use the notion of Clarke sub-derivative introduced earlier
  in Definition \ref{Clarke}.

Let us recall that by Proposition \ref{hausdorff},
if $g_n : \T\to \R$ are equi-semi-concave functions  converging to $g : \T \to \R$, then $\mathcal{PG} (g_n')$ converges to $\mathcal{PG} ( g')$ for the Hausdorff distance.

Let us now state our result:
\begin{propos}\label{discount}
Let $f:\A \to \A$ be a symplectic twist map. For $c\in \R$, we denote by $\cu_c$ the weak K.A.M. discounted solution. Then the map $c\mapsto \mathcal{PG} ( \cu_c')$ is continuous.
\end{propos}
As a straightforward corollary, we deduce for instance that if $c_n \to c$ and $x_n \to x$ and if  the $\cu'_{c_n}(x_n)$ exist, as well as $\cu'(c)(x)$, then $\cu'_{c_n}(x_n) \to \cu'(c)(x)$.
 
 
%

%

\begin{proof}[proof of Proposition \ref{discount}]
If $\rho(c_0)\in \R\setminus \Q$, there is a unique weak K.A.M. solution up to constants, hence continuity of $\mathcal{PG} ( \cu_c')$ at $c_0$ follows from Proposition \ref{hausdorff}.

If $\rho(c) = r\in \Q$, let us denote $\rho^{-1}(r) = [c_1,c_2]$. Again, continuity at $c_1$ and $c_2$ is obvious as there is a unique weak K.A.M solution at these cohomology classes (see Proposition \ref{unicite}).

It remains to study what happens inside $(c_1,c_2)$ and we will prove that in this interval, the map $c\mapsto \cu_c$ is concave. Let us set $M $ the set of Mather measures corresponding to any cohomology class $c\in (c_1,c_2)$. Recall that as seen in \eqref{Matherinv} page \pageref{Matherinv}, this set does not depend on $c$. Moreover, the function $\alpha$ is affine on $(c_1,c_2)$. 

From \cite{DFIZ2}, we know that $\cu_c(x) = \sup_{u} u(x)$, where the supremum is taken amongst (continuous) $c$-dominated functions $u:\T\to \R$ such that $\int u(x) d\mu(x,y) \leq 0$ for all $\mu \in M$. Moreover, it is proven that $\int \cu_c(x) d\mu(x,y) \leq 0$  for all $\mu \in M$. Let now $c,c' \in (c_1,c_2)$ and $\lambda \in [0,1]$. Let us set $v= \lambda \cu_c + (1-\lambda)\cu_{c'}$.

As $\int \cu_c(x) d\mu(x,y) \leq 0$ and $\int \cu'_c(x) d\mu(x,y) \leq 0$   for all $\mu \in M$ we deduce that $\int v(x) d\mu(x,y) \leq 0$  for all $\mu \in M$.

Moreover, passing to lifts (with the same $\sim$ notation as previously), from

$$\forall \theta, \theta'\in \R, \quad  \widetilde \cu_c(\theta)-\widetilde \cu_c(\theta')\leq S(\theta', \theta)+c(\theta'-\theta)+ \alpha(c);$$
$$\forall \theta, \theta'\in \R, \quad  \widetilde \cu_{c'}(\theta)-\widetilde \cu_{c'}(\theta')\leq S(\theta', \theta)+c'(\theta'-\theta)+ \alpha(c');$$
and recalling that $\alpha\big(\lambda c + (1-\lambda)c'\big) = \lambda \alpha(c) + (1-\lambda)\alpha(c')$, we get

$$\forall \theta, \theta'\in \R, \quad  \tilde v(\theta)-\tilde v(\theta')\leq S(\theta', \theta)+\big(\lambda c + (1-\lambda)c'\big)(\theta'-\theta)+\alpha\big(\lambda c+(1-\lambda) c'\big).$$
Hence $v$ is $\big(\lambda c+(1-\lambda) c'\big)$-dominated. We conclude that $v\leq \cu_{\lambda c + (1-\lambda) c'}$, proving the claim, and the Proposition.

\end{proof}
 \begin{remk}
 The previous proof is intimately linked to the $1$-dimensional setting we work with. Indeed, it was communicated to us by Patrick Bernard that as soon as we move up to dimension $2$, there are examples on $\T^2$ for which it is not possible to construct a function $c\mapsto u_c$ that maps to each cohomology class a weak K.A.M. solution and such that $c\mapsto Du_c$ is continuous (in any possible way). 
 \end{remk}
We obtain as a corollary:

\begin{cor}
The function $\cu(x,c) = \cu_c(x) - \cu_c(0)$ also satisfies the conclusions of Theorem \ref{Tgenecont}.
\end{cor}

We now  give a $C^\infty$ integrable example for which the discounted method doesn't select a transversely continuous weak K.A.M. solution. 
 
 \begin{exa}
 We use the notation of Theorem \ref{Tgeneder}. We define $F_0, H:\A\rightarrow \A$ by $F_0(\theta, r)=(\theta +r, r)$ and $H(\theta, r)=(h(\theta), \frac{r}{h'(\theta)})$ where $h:\T\rightarrow \T$ is a smooth orientation preserving diffeomorphism of $\T$ such that $h(t)=t+d(t)$ and  $d:\T\rightarrow \R$ satisfies $d(0)=0$ and 
 \begin{equation}\label{Edemi}\int_\T d(t)dt>\frac{d(\frac{1}{2})}{2}.
 \end{equation}  Observe that $h^{-1}(t)=t-d\circ h^{-1}(t)$.  As the symplectic diffeomorphism $H$ maps a vertical $\{ \theta\}\times \R$ onto a vertical $\{ h(\theta)\}\times \R$ and preserves the transversal orientation, the smooth diffeomorphism\footnote{{Note that $F_0$ is the time-$1$ map of the Hamiltonian function $f_0(\theta, r) = \frac 12 r^2$. It follows that $F$, being conjugated to $F_0$ by a symplectic map, is itself the time-$1$ map of the Tonelli Hamiltonian $f_0\circ H^{-1}$. }}  $F=H\circ F_0\circ H^{-1}$ is also a symplectic  $C^\infty$ integrable twist diffeomorphism. The new invariant foliation is the set of the graphs of $\eta_c(\theta)=\frac{c}{h'\textrm{$\big($}h^{-1}(\theta)\textrm{$\big)$}}=c(h^{-1})'(\theta)$. Hence we have $u_c(\theta)=-cd\circ h^{-1}(\theta)$. Observe that the function $u$ is smooth.\\
Then $H_c(\theta)=\theta+\frac{\partial u_c}{\partial c}(\theta)=\theta-d\circ h^{-1}(\theta)=h^{-1}(\theta)$. Hence the measure defined on $\T$ by $\mu ([0, \theta])=h^{-1}(\theta)$, i.e. the measure with density $\frac{1}{h'\circ h^{-1}}$, is invariant by the restricted-projected Dynamics $g_c$. When the rotation number $ \rho(c)$ of $g_c$ is irrational, this is the only measure invariant by $g_c$.\\
 Let us recall that the discounted solution $\cu_c$ that is selected in \cite{SuThi} and  \cite{DFIZ2} is the weak K.A.M. solution that is the supremum of the subsolutions that satisfy for every   minimizing $g_c$-invariant measure $\mu$: $\int u_cd\mu\leq 0$. When $c$ is irrational, we deduce that
 $$\cu_c(\theta)=u_c(\theta)-\int u_c(t)d\mu(t)=c\left(\int_\T  d\circ h^{-1}(t)(h^{-1})'(t)dt -d\circ h^{-1}(\theta)\right);$$
 i.e. 
\begin{equation}\label{Eirrat} \cu_c(\theta)=c\left(\int_\T d(t)dt-d\circ h^{-1}(\theta)\right)=u_c(\theta)+c\int_\T d(t)dt.\end{equation}
  Assume now that $c=\frac{1}{2}$. Then 
  $$g_{\frac{1}{2}}(0)=h\circ R_{\frac{1}{2}}\circ h^{-1}(0)=h\Big(\frac{1}{2}\Big)=\frac{1}{2}+d\Big(\frac{1}{2}\Big)\quad{\rm and}\quad g_{\frac{1}{2}}\left( \frac{1}{2}+d\Big(\frac{1}{2}\Big)\right)=0.$$
  The mean of the two Dirac measures
  $$\nu=\frac{1}{2}\left( \delta_0+\delta_{\frac{1}{2}+d(\frac{1}{2})}\right)$$
  is a measure that is invariant by $g_\frac{1}{2}$. Hence $\cu_{\frac{1}{2}}(\theta)=u_{\frac{1}{2}}(\theta)-K$ with $K\geq\int_\T u_{\frac{1}{2}}d\nu$. We deduce that $$K\geq \frac{1}{2}\left( u_{\frac{1}{2}}(0)+u_{\frac{1}{2}}\bigg(\frac{1}{2}+d\Big(\frac{1}{2}\Big)\bigg)\right)=-\frac{1}{4}\left( d\circ h^{-1}(0)+d\circ h^{-1} \bigg(\frac{1}{2}+d\Big(\frac{1}{2}\Big)\bigg)\right);$$
  i.e.
  $$K\geq -\frac{1}{4}d\Big(\frac{1}{2}\Big).$$
By  Inequality  (\ref{Edemi}), we know that $\varepsilon=\int_\T d(t)dt-\frac{d(\frac{1}{2})}{2}>0$.  We have then 
  $$\cu_{\frac{1}{2}}(\theta)\leq u_{\frac{1}{2}}(\theta)+ \frac{1}{4}d\Big(\frac{1}{2}\Big)= u_{\frac{1}{2}}(\theta)+\frac{1}{2}\int_\T d(t)dt-\frac{\varepsilon}{2}
  $$
  Using Equation (\ref{Eirrat}), we deduce that
  $$\limsup_{c\rightarrow\frac{1}{2}}\cu_c(\theta)\geq \cu_{\frac{1}{2}}(\theta)+\frac{\varepsilon}{2}.$$
  Hence $(\theta, c)\mapsto \cu_c(\theta)$ is not continuous.\\
  Observe that in the integrable case, there exists a unique weak K.A.M. solution in each cohomology class up to the addition of a constant. Hence selecting a weak K.A.M. solution in every cohomology class is reduced in this case to choosing a constant. Using this remark, it can be proved that for the integrable case, the discounted choice is lower semi-continuous.
 \end{exa}

\subsection{A  foliation by graphs that is the inverse image of the standard foliation by a symplectic map but not by a symplectic  homeomorphism}
 We will use two special functions
 \begin{itemize}
 \item $\gamma:\T\rightarrow \R$ a $C^\infty$ function such that $\gamma'_{[\frac{1}{2}-\varepsilon, \frac{1}{2}+\varepsilon]}=-1$ and 
 $\gamma'_{\T\backslash [\frac{1}{2}-\varepsilon, \frac{1}{2}+\varepsilon]}>-1$;
 \item $\zeta:\R\rightarrow \R$ a $C^\infty$ function that is increasing, such that $\zeta'(0)=1$ and $\zeta'_{\R\backslash\{ 0\}}<1$ with $\displaystyle{\lim_{\pm\infty}\zeta'=\frac{1}{2}}$.
 \end{itemize}
 The function $u(\theta, c)=\zeta(c)\gamma(\theta)$ defines the foliation in graphs of $$\eta_c=c+\frac{\partial u}{\partial\theta}=c+\zeta(c)\gamma'.$$
 The derivative with respect to $c$ of $\eta_c(\theta)$ is then $\frac{\partial \eta_c}{\partial c}(\theta)=1+\zeta'(c)\gamma'(\theta)$ that is non negative, vanishes only for $(\theta, c)\in [\frac{1}{2}-\varepsilon, \frac{1}{2}+\varepsilon]\times \{0\}$ and is larger that $\frac{1}{3}$ close to $\pm\infty$. Hence every map $c\in\R\mapsto \eta_c(\theta)\in\R$ is a homeomorphism and we have indeed a $C^0$ foliation.\\
 Let us introduce $h_c(\theta)=\theta+\frac{\partial u}{\partial c}(\theta)=\theta+\gamma(\theta)\zeta'(c)$. Its derivative is $1+\zeta'(c)\gamma'(\theta)$  that is non negative and vanishes only if $(\theta, c)\in [\frac{1}{2}-\varepsilon, \frac{1}{2}+\varepsilon]\times \{0\}$. Hence $h_0$ is not a homeomorphism but all the other $h_c$ are homeomorphisms. \\
 We deduce from Theorem \ref{TC0arn} that this foliation is not symplectically homeomorphic to the standard one.
 
 We will now prove that the map defined by $H\big(\theta, \eta_c(\theta)\big)=(h_c(\theta), c)$ is a symplectic map, i.e. the limit (for the $C^0$ topology) of 
a sequence  of symplectic diffeomorphisms.\\
Let $\gamma_n:\T\rightarrow \R$ be a sequence of $C^\infty$ maps that converges to $\gamma$ in $C^1$ topology and satisfies $\gamma_n'>-1$. Let $(\zeta_n)$ be a sequence of $C^\infty$ diffeomorphisms of $\R$ that $C^1$ converges to $\zeta$ and satisfies $\zeta_n'<1$. We introduce $u_n(\theta, c)=\gamma_n(\theta)\zeta_n(c)$. Then $\eta_c^n(\theta)=c+\zeta_n(c)\gamma_n'(\theta)$ defines a smooth foliation, $h_c^n(\theta)=\theta+\gamma_n(\theta)\zeta_n'(c)$ is a smooth diffeomorphism of $\T$ and
$$K_n(\theta, c)=\left( \left( h_c^n\right)^{-1}(\theta), \eta^n_c\left(\big( h_c^n\right)^{-1}(\theta)\big)\right)$$
is a symplectic smooth diffeomorphism that maps the standard foliation  to the foliations by the graphs of $\left( \eta_c^n\right)_{c\in\R}$.\\
If $H_n=K_n^{-1}$, observe that $H_n=G_n\circ F_n^{-1}$ where 
\begin{itemize}
\item $F_n(\theta, c)=\big(\theta, c+\frac{\partial u_n}{\partial \theta}(\theta, c)\big)$ converges uniformly to $F(\theta, c)=\big(\theta, c+\frac{\partial u}{\partial \theta}(\theta, c)\big)$;
\item $G_n(\theta, c)=(\theta+\frac{\partial u_n}{\partial c}(\theta, c), c)$ converges uniformly to $G(\theta, c)=(\theta+\frac{\partial u}{\partial c}(\theta, c), c)$.
\end{itemize}
Finally, $H_n=G_n\circ F_n^{-1}$ converges uniformly to $H=G\circ F^{-1}$

 \section{Some results concerning the full pseudographs}\label{Apfulpseudo}

 Most of the results that follow are standard and even hold in all dimension. One can find them in similar of different formulations in \cite{cansin}. However, we provide proofs for the reader's convenience.

 \subsection{An equivalent definition}
\hglue 14truecm 
 
 \begin{defi}
 Let $u:\R\rightarrow \R$ be a $K$ semi-concave function. Then $p\in\R$ is a {\em super-derivative} of $u$ at $x\in\R$ if
$$\forall y\in \R,\quad u(y)-u(x)-p(y-x)\leq \frac{K}{2}(y-x)^2.$$
We denote the set of super-derivatives of $u$ at $x$ by $\partial^+u(x)$.   It is a convex set.
\end{defi}
Observe that a derivative is always a super-derivative.  If $u:\R\rightarrow \R$ is $K$-semi-concave, then  $x\mapsto u(x)-\frac{K}{2}x^2$ is concave and thus  locally Lipschitz, and $x\mapsto u'(x)-Kx$ is non-increasing. Hence a $1$-periodic $K$-semi-concave function is $K$-Lipschitz.

 Observe also that $\displaystyle{\bigcup_{x\in \T}\{ x\}\times \partial^+u(x)}$ is compact.
\begin{propos}
 Let $u:\R\rightarrow \R$ be a $K$-semi-concave function. Then, for every $x\in\R$, we have 
 $$\partial u (x)=\{ x\}\times \partial ^+u(x).$$
\end{propos} 
 Hence the full pseudograph of $u$ is also the subbundle of all the super-derivatives of $u$.

 \begin{proof}
Let us prove the inclusion $\partial u (x)\subset\{ x\}\times \partial ^+u(x)$.  Let us consider $(x, p)\in\partial u (x)$.  Then there exist $(x,p_-), (x, p_+)\in \overline{\cg(u')}$ such that $p_-\leq p\leq p_+$ and there exist two sequences $(x_n, p_n), (y_n, q_n)\in\cg(u')$ that respectively converge to $(x, p_-)$, $(x,p_+)$. Every derivative is a super-derivative and a limit of super-derivatives is a super-derivative. Hence, we have $p_-, p_+\in \partial^+u(x)$. By convexity of $\partial^+u(x)$, we deduce that $p\in \partial^+u(x)$.


Let us now prove the reverse inclusion. Being $K$-semi-concave, $u$ is $K$-Lipschitz, hence the set of all its super-derivatives is bounded (by $K$). If $x\in \R$, we have then $\partial^+u(x)=[p_-, p_+]$ with $-K\leq p_-\leq p_+\leq K$. We will prove that $(x,p_-), (x, p_+)\in \partial u(x)$. We have
\begin{multline*}
\forall y\in \R,\quad  u(y)-u(x)-p_-(y-x)\leq \frac{K}{2}(y-x)^2
\\
{\rm and}\quad u(y)-u(x)-p_+(y-x)\leq \frac{K}{2}(y-x)^2.
\end{multline*}
This implies that
\begin{itemize}
\item for $y>x$, we have 
$$\frac{u(y)-u(x)}{y-x}\leq p_-+\frac{K}{2}(y-x);$$
\item for $y<x$, we have 
$$\frac{u(y)-u(x)}{y-x}\geq p_++\frac{K}{2}(y-x).$$
\end{itemize}
Recall that $\frac{u(y)-u(x)}{y-x}=\frac{1}{y-x}\int_x^yu'(t)dt$. This gives the existence of two sequences $(x_n)\in (-\infty, x)$ and $(y_n)\in (x, +\infty)$ that converge to $x$ where $u$ is differentiable and
$$\limsup u'(x_n)\geq p_+\quad{\rm and}\quad \liminf u'(y_n)\leq p_-.$$
As we know that a derivative is a super-derivative, that the set of super-derivatives is closed and that $\partial^+u(x)=[p_-, p_+]$, we deduce that 
$$\big(x, \lim  u'(x_n)\big)=(x,  p_+)\in \partial u(x) \quad{\rm and}\quad \big(x, \lim  u'(y_n)\big)=(x,  p_-)\in \partial u(x).$$
 \end{proof}
 \subsection{Proof of Lemma \ref{Lfullman}}\label{ssLfullman}
 We just recall the argument of the proof of
 \begin{lemma}
For all $c\in \R$, $\mathcal{PG}(c+u'_c)$ is a Lipschitz one dimensional compact manifold that is an essential circle. 
\end{lemma}
\begin{proof}
It is proved in  \cite{Arna2}, that for every $c\in \R$ and every $K$-semi-concave function $u:\T\rightarrow \R$, there exists $\tau>0$ such that $\varphi_{-\tau}\big(\mathcal{PG}(c+u')\big)$ is the graph of a Lipschitz function, where $(\varphi_t)$ is the flow of the pendulum. This gives the wanted result.
\end{proof}

 \subsection{Proof of Proposition \ref{hausdorff}}\label{AppB3}
 Let us now prove the following proposition\footnote{{The statement holds in arbitrary dimension and follows from the same result for concave functions. We present here a simple proof relying on the $1$-dimensional setting.}}. 
 
 \begin{propos}
Let $(f_n)_{n\in \mathbb N} $ be a sequence of equi-semi-concave functions from $\T$ to $\R$ that converges (uniformly) to a function $f$ (that is hence also semi-concave).

Then $\big(\mathcal{PG}(f'_n) \big)$ converges to $\mathcal{PG}(f')$ for the Hausdorff distance.
\end{propos}

\begin{proof}
Let us prove that the lim sup of the $\mathcal{PG}(f'_n) $ is in $\mathcal{PG}(f')$. Up to a subsequence, we consider $(x_n, p_n)\in \mathcal{PG}(f'_n) $ that converges to some $(x,p)$, and we want to prove that $(x,p)\in \mathcal{PG}(f')$. We have
$$\forall n, \forall y\in \R, \quad f_n(y)-f_n(x_n)-p_n(y-x_n)\leq \frac{K}{2}(y-x_n)^2.$$
Taking the limit, we deduce that $(x,p)\in \mathcal{PG}(f')$.

Let us now assume that $\big(\mathcal{PG}(f'_n) \big)$ doesn't converge to $\mathcal{PG}(f')$. There exists a point $(x, p)\in \mathcal{PG}(f')$,  $r>0$ and $N\geq 1$ such that, up to a subsequence, 
$$\forall n\geq N, \quad\mathcal{PG}(f'_n)\cap B\big((x, p), r\big)=\varnothing.$$
Hence, for $n$ large enough, $\mathcal{PG}(f'_n)$ is contained in a small neighbourhood of a simple arc (and not loop).
This implies that for $n$ large enough, $\mathcal{PG}(f'_n)$ doesn't separate the annulus into two unbounded connected components, a contradiction.

\end{proof}

 \section{Green bundles}\label{ssGreenb} Here we recall the theory of Green bundles. More details or proofs can be found in \cite{Arna3}.
We fix a lift $F$ of a conservative twist map $f$.  
\begin{notas}\label{Nota}{\rm
\begin{enumerate}
\item[$\bullet$] $V(x)=\{ 0\}\times\R\subset T_x\R^2$ and for $k\not= 0$, we have $G_k(x)=DF^k(F^{-k}x)V(f^{-k}x)$;
\item[$\bullet$] the slope of $G_k$ (when defined) is denoted by $s_k$: $$G_k(x)=\{ (\delta\theta, s_k(x)\delta\theta); \ \ \delta\theta\in\R\};$$
\item[$\bullet$] if $\gamma$ is a real Lipschitz function defined on $\T$ or $\R$, then 
$$\gamma'_+(x)=\limsup_{\substack{y,z\rightarrow x\\ y\not=z}}\frac{\gamma(y)-\gamma(z)}{y-z}\quad{\rm and}\quad \gamma'_-(t)=\liminf_{\substack{y,z\rightarrow x\\ y\not=z}}\frac{\gamma(y)-\gamma(z)}{y-z}.$$
\end{enumerate}}
\end{notas}
Then  
\begin{enumerate}
\item if the orbit of $x\in\R^2$ is minimizing,  we have 
$$\forall n\geq 1,\quad s_{-n}(x)<s_{-n-1}(x)<s_{n+1}(x)<s_n(x);$$
\item in this case,   the two {\em Green bundles} at $x$ are $G_+(x), G_-(x)\subset T_x(\R^2)$ with slopes $s_-$, $s_+$ where $\displaystyle{s_+(x)=\lim_{n\rightarrow +\infty}s_n(x)}$ and $\displaystyle{ s_-(x)=\lim_{n\rightarrow +\infty}s_{-n}(x)}$;
\item the two Green bundles  are invariant under $Df$: $Df(G_\pm)=G_\pm\circ f$;
\item we have $s_+\geq s_-$;
\item the map $s_-$ is lower semi-continuous and the map  $s_+$ is upper semi-continuous;
\item hence $\{ G_-=G_+\}$ is a  $G_\delta$ subset of the set of points whose orbit is minimizing (this last set  is a closed set) and   $s_-=s_+$ is continuous at every point of this set.
\end{enumerate}

Let us focus on the case of an invariant curve that is the graph of $\gamma$. Then we have
\begin{propos}\label{PGreensand}
Assume that the graph of $\gamma\in C^0(\T, \R)$ is invariant by $F$. Then the orbit of any point contained in the graph of $\gamma$ is minimizing and we have
$$\forall \theta\in\T,\quad s_-\big(\theta, \gamma(\theta)\big)\leq\gamma'_-(\theta)\leq \gamma'_+(\theta)\leq s_+\big(\theta, \gamma(\theta)\big).$$
\end{propos}

\begin{propos}\label{Pdyncrit} {\bf (Dynamical criterion)} Assume that $x$ has its orbit that is minimizing and that is contained in some strip $\R\times[-K,K]$ (for example $x$ is in some invariant graph) and that $v\in T_x\R^2\backslash\{ 0\}$. Then
\begin{enumerate}
\item[$\bullet$] if $\displaystyle{\liminf_{n\rightarrow +\infty} |D(\pi\circ F^n)(x)v|<+\infty}$, then $v\in G_-(x)$;
\item[$\bullet$] if $\displaystyle{\liminf_{n\rightarrow +\infty} |D(\pi\circ F^{-n})(x)v|<+\infty}$, then $v\in G_+(x)$.
\end{enumerate}
\end{propos} 
In particular, if the Dynamics restricted to some invariant graph is totally periodic, then along this graph we have $G_-=G_+$ and the graph is $C^1$. The $C^1$ property can also  be proved  by using the implicit functions theorem.

\section{Sketch of the proof of point \ref{K.A.M.ext} page \pageref{K.A.M.ext}}\label{appendix-3}

We wish to explain why if $u:\cm\big(\rho(c)\big)\rightarrow \R$ is dominated, then there exists only one extension $U$ of $u$ to $\T$ that is a weak K.A.M. solution for $\widehat T^c$ that is given by  
$$\forall x\in \T, \quad U(x) = \inf_{y\in  \cm\textrm{$\big($}\rho(c)\textrm{$\big)$}} u(y) + S^c(y,x).$$

\begin{itemize}
\item It is a general fact that if $y\in \cm\textrm{$\big($}\rho(c)\textrm{$\big)$}$ the function $x\mapsto S^c(y,x)$ is a weak K.A.M solution that vanishes at $x=y$ (see \cite[Definition 2.1 and Proposition 2.8]{Za1} recalling that the function $S^c$ corresponds to the Ma\~n\' e potential $\varphi$ in the reference and that our Mather set $\cm\textrm{$\big($}\rho(c)\textrm{$\big)$}$ is included in the Aubry set). As the set of weak K.A.M. is invariant by addition of constants and an infimum of weak K.A.M. solutions is a weak K.A.M. solution (\cite[Lemma 2.33]{Za1}) it follows that $U$ is a weak K.A.M. solution.
\item To prove that $U=u$ on $\cm\textrm{$\big($}\rho(c)\textrm{$\big)$}$ just notice that as $u$ is dominated, if $x\in \cm\big(\rho(c)\big)$,
$$\forall y \in \cm\big(\rho(c)\big),\quad u(y)+S^c(y,x) \geq u(x) = u(x) +S^c(x,x).$$
\item It remains to prove that $U$ is unique. This follows from the fact that if two weak K.A.M. solutions $U_1$ and $U_2$  coincide on $\cm\big(\rho(c)\big)$ they are equal.

Let $x_0\in \T$. One constructs inductively a sequence $(x_n)_{n\leq 0}$ such that 
$$\forall n< 0, \quad U_1(x_0)=U_1(x_n)+\sum_{k=n}^{-1}S^c(x_k,x_{k+1}).$$
As $U_2$ is a weak K.A.M. (hence dominated) one also has 
$$\forall n< 0, \quad U_2(x_0)\leq U_2(x_n)+\sum_{k=n}^{-1}S^c(x_k,x_{k+1}).$$
Hence $U_2(x_0)-U_1(x_0) \leq U_2(x_n) - U_1(x_n)$. To conclude, one proves, using a Krylov-Bogoliubov type argument that there exists a subsequence $(x_{\varphi(n)})$ that converges to a point $x\in \cm\big(\rho(c)\big)$, hence proving that $U_2(x_0)-U_1(x_0) \leq 0$. Then the result follows by a symmetrical argument.
\end{itemize}

\bibliographystyle{amsplain}

\end{document}